\documentclass[12pt]{amsart}

\usepackage{amssymb, amsmath, amscd, newlfont, bbm, slashbox, etoolbox}
\usepackage[foot]{amsaddr}

\usepackage{hyperref} 
\hypersetup{
	colorlinks=true,        
	linkcolor={black},  
	citecolor={black},
	urlcolor={black}     
}
\usepackage{enumitem}
\usepackage{tikz-cd}
\usepackage{comment}
\usepackage{cite, mathtools}

\newcommand{\spacedcdot}{{\,\cdot\,}}

\newcommand{\Q}{{\mathbb{Q}}}
\newcommand{\C}{{\mathbb{C}}}
\newcommand{\R}{{\mathbb{R}}}
\newcommand{\Z}{{\mathbb{Z}}}

\newcommand{\GL}{{\mathrm{GL}}}
\newcommand{\SU}{{\mathrm{SU}}}
\newcommand{\Sp}{{\mathrm{Sp}}}
\newcommand{\SO}{{\mathrm{SO}}}
\newcommand{\SL}{{\mathrm{SL}}}
\newcommand{\U}{{\mathrm{U}}}

\DeclareMathOperator{\vol}{{\mathrm{vol}}}
\DeclareMathOperator{\artanh}{{\mathrm{artanh}}}

\newcommand{\calD}{{\mathcal{D}}}

\newcommand{\calH}{{\mathcal{H}}}

\newcommand{\calL}{{\mathcal{L}}}

\newcommand{\calR}{{\mathcal{R}}}

\newcommand{\calU}{{\mathcal{U}}}

\usepackage[mathscr]{eucal}

\providecommand{\abs}[1]{\left\lvert#1\right\rvert}
\providecommand{\norm}[1]{\left\lVert#1\right\rVert}
\providecommand{\scal}[2]{\left<#1,#2\right>}

\newtheorem{theorem}{Theorem}[section]
\newtheorem{lemma}[theorem]{Lemma}
\newtheorem{proposition}[theorem]{Proposition}
\newtheorem{corollary}[theorem]{Corollary}

\theoremstyle{definition}

\newtheorem{example}[theorem]{Example}

\theoremstyle{remark}
\newtheorem{remark}[theorem]{Remark}
\newtheorem*{remark*}{Remark}

\newtheoremstyle{named}{}{}{\itshape}{}{\bfseries}{.}{.5em}{#1 \thmnote{#3}}
\theoremstyle{named}

\numberwithin{equation}{section}

\title[Poincar\'e series on irreducible bounded symmetric domains]{On the non-vanishing of Poincar\'e series on irreducible bounded symmetric domains}

\author{Sonja \v Zunar}

\address{University of Zagreb,
	Faculty of Geodesy,
	Ka\v ci\'ceva 26,
	10000 Zagreb,
	Croatia}
\email{szunar@geof.hr}

\thanks{This work is supported by the Croatian Science Foundation under the project number HRZZ-IP-2022-10-4615.}

\patchcmd\newpage{\vfil}{}{}{}

\begin{document}

\begin{abstract}
	Let $ \mathcal D\equiv G/K $ be an irreducible bounded symmetric domain. Using a vector-valued version of Mui\'c's integral non-vanishing criterion for Poincar\'e series on locally compact Hausdorff groups, we study the non-vanishing of holomorphic automorphic forms on $ \mathcal D $ that are given by Poincaré series of polynomial type and correspond via the classical lift to the Poincaré series of certain $ K $-finite matrix coefficients of integrable discrete series representations of $ G $. We provide an example application of our results in the case when $ G=\SU(p,q) $ and $ K=\mathrm S(\U(p)\times\U(q)) $ with $ p\geq q\geq1 $.
\end{abstract}

\maketitle

\section{Introduction}

Holomorphic automorphic forms on bounded symmetric domains have long sparked the interest of researchers in representation theory and number theory (see, e.g., \cite{garrett83, foth07, alluhaibi_barron19}). In \cite{cartan35}, \'E.\ Cartan proved that every irreducible bounded symmetric domain $ \mathcal D $ is biholomorphic to an irreducible Hermitian symmetric space $ G/K $, where $ G $ is one of the Lie groups $ \SU(p,q) $, $ \Sp(n,\R) $, $ \SO^*(2n) $, and $ \SO_\circ(n,2) $ with suitable $ p,q,n\in\Z_{>0} $ or one of the exceptional Lie groups $ E_{6(-14)} $ and $ E_{7(-25)} $, and $ K $ is a maximal compact subgroup of $ G $ \cite[Ch.\,X, \S6.3]{helgason78}. In the case when $ G=\Sp(n,\R) $, the domain $ \mathcal D $ is biholomorphic to the Siegel upper half-space $ \mathcal H_n $ of degree $ n $. Other types of irreducible bounded symmetric domains have analogous unbounded realizations as Siegel domains of the first and second kind \cite{pyatetskii-shapiro69}.

In \cite{zunar24}, we used G.\ Mui\'c's integral non-vanishing criterion for Poincar\'e series on locally compact Hausdorff groups to study the non-vanishing of vector-valued Siegel cusp forms on $ \mathcal H_n $ that lift in a standard way to the Poincar\'e series of certain $ K $-finite matrix coefficients of integrable discrete series representations of $ \Sp(n,\R) $. In this paper, we adapt the methods of \cite{zunar24} to study the non-vanishing of holomorphic automorphic forms on a general irreducible bounded symmetric domain $ \mathcal D $ that are given by Poincar\'e series of polynomial type and lift to Poincar\'e series of $ K $-finite matrix coefficients of integrable discrete series representations of a suitable Lie group $ G $.

Let us give a brief overview of our main results. We work with the Harish-Chandra realization of irreducible bounded symmetric domains, defined as follows. Let $ \mathfrak g $ be a simple, non-compact real Lie algebra with Cartan decomposition $ \mathfrak g=\mathfrak k\oplus\mathfrak p $ and a compact Cartan subalgebra $ \mathfrak h $. Writing $ \mathfrak a_\C=\mathfrak a\otimes_\R\C $ for a real Lie algebra $ \mathfrak a $, we fix a positive root system $ \Delta^+ $ for $ \mathfrak g_\C $ with respect to $ \mathfrak h_\C $ such that every $ \alpha $ in the set $ \Delta_n^+ $ of non-compact positive roots is totally positive (see Section \ref{sec:104}). Let $ G_\C $ be the simply connected complex Lie group with Lie algebra $ \mathfrak g_\C $. Let $ G $, $ K $, $ K_\C $, $ P_\C^+ $, and $ P_\C^- $ be the connected Lie subgroups of $ G_\C $ with Lie algebras $ \mathfrak g $, $ \mathfrak k $, $ \mathfrak k_\C $, $ \mathfrak p_\C^+=\sum_{\alpha\in\Delta_n^+}\mathfrak g_\alpha $, and $ \mathfrak p_\C^-=\sum_{\alpha\in\Delta_n^+}\mathfrak g_{-\alpha} $, respectively. Every $ g\in P_\C^+K_\C P_\C^- $ has a unique factorization $ g=g_+g_0g_- $ with $ g_+\in P_\C^+ $, $ g_0\in K_\C $, and $ g_-\in P_\C^- $. We have $ GK_\C P_\C^-\subseteq P_\C^+K_\C P_\C^- $, and the quotient $ G/K\equiv GK_\C P_\C^-/K_\C P_\C^- $ identifies with the irreducible bounded symmetric domain 
\[ \mathcal D=\left\{\exp^{-1}(g_+):g\in G\right\}\subseteq\mathfrak p_\C^+, \]
where $ \exp $ is the exponential map $ \mathfrak p_\C^+\to P_\C^+ $.
Under this identification, the canonical left action of $ G $ on $ G/K $ transforms into the action of $ G $ on $ \mathcal D $ given by
\[ g.x=\exp^{-1}((g\exp x)_+),\quad g\in G,\ x\in\mathcal D. \]

Given $ \alpha\in\Delta^+ $, there exists a unique $ H_\alpha\in[\mathfrak g_\alpha,\mathfrak g_{-\alpha}] $ such that $ \alpha(H_\alpha)=2 $. We fix an analytically integral linear functional $ \Lambda $ on $ \mathfrak h_\C $ such that $ \Lambda(H_\alpha)\geq0 $ for all compact roots $ \alpha\in\Delta^+ $ and 
\[ (\Lambda+\delta)(H_\alpha)<\textstyle-\frac12\sum_{\beta\in\Delta^+}\abs{\beta(H_\alpha)} \]
for all $ \alpha\in\Delta_n^+ $, where $ \delta=\frac12\sum_{\beta\in\Delta^+}\beta $.
It is well known that the irreducible unitary representation  $ (\rho,V) $ of $ K $ of highest weight $ \Lambda $ extends to an irreducible holomorphic representation of $ K_\C $ and occurs with multiplicity one in the integrable discrete series representation $ (\pi_\rho,\calH_\rho) $ of $ G $ of highest weight $ \Lambda $ (see Lemma \ref{lem:022}).

Let $ \Gamma $ be a discrete subgroup of finite covolume in ${} G $ such that $ \pi_\rho $ occurs in the right regular representation $ \left(R, L^2(\Gamma\backslash G)\right) $ of $ G $ with finite multiplicity. This condition is satisfied, e.g., if $ G $ and $ \Gamma $ satisfy one of the following (see Remark \ref{rem:152}):
\begin{enumerate}
	\item The quotient space $ \Gamma\backslash G $ is compact.
	\item $ G=\mathcal G(\mathbb R) $ for a connected semisimple algebraic group $ \mathcal G $ defined over $ \Q $, and $ \Gamma $ is an arithmetic subgroup of $ \mathcal G(\Q) $.
\end{enumerate}
 
We consider the space $ \calH_\rho^\infty(\Gamma) $ of holomorphic automorphic forms of weight $ \rho $ for $ \Gamma $ on $ \mathcal D $, defined as the space of holomorphic functions $ f:\mathcal D\to V $ with the following two properties:
\begin{enumerate}[label={(\arabic*)}]
	\item $ f(\gamma.x)=\rho(\left(\gamma\exp x\right)_0)f(x) $ for all $ \gamma\in\Gamma $ and $ x\in\mathcal D $.
	\item $ \sup_{x\in \mathcal D}\norm{\rho(\textbf{\textit x}_0)^{-1}f(x)}_V<\infty $, where $ \textbf{\textit x}\in G/K $ is such that $ \textbf{\textit x}.0=x $.
\end{enumerate}
We give a representation-theoretic proof of the following classical result (see Theorem \ref{thm:100a}).

\begin{theorem}\label{thm:100}
	The space $ \calH_\rho^\infty(\Gamma) $ is the space of (absolutely and locally uniformly convergent) Poincar\'e series
	\[ \left(P_{\Gamma,\rho}f\right)(x)=\sum_{\gamma\in\Gamma}\rho(\left(\gamma\exp x\right)_0)^{-1}f(\gamma.x),\quad x\in\mathcal D, \]
	where $ f $ goes over the space $ \mathcal P(\calD,V) $ of polynomial functions $ \mathcal D\to V $.
\end{theorem}

Let us briefly describe our proof of Theorem \ref{thm:100}, which is based on the methods of \cite[proof of Theorem 1-1]{muic10} and \cite[proof of Theorem 1.1]{zunar24}.
Let $ v_\Lambda\in V $ be a highest weight vector for $ \rho $ such that $ \norm{v_\Lambda}_V=1 $. Given $ f:\mathcal D\to V $, let $ F_f:G\to V $,
\[ F_f(g)=\rho(g_0)^{-1}f(g.0), \]
and let $ v_\Lambda^*F_f=\scal{F_f(\spacedcdot)}{v_\Lambda}_{V}:G\to\C $.
We show that the assignment $ f\mapsto v_\Lambda^*F_f $ defines a linear isomorphism $ \Phi_{\rho,\Gamma} $ from $ \calH_\rho^\infty(\Gamma) $ to the space $ \calL_\rho^\infty(\Gamma)=\left(L^2(\Gamma\backslash G)_{[\pi_\rho^*]}\right)_{-\Lambda} $
of weight $ -\Lambda $ vectors in the $ \pi_\rho^* $-isotypic component of the right regular representation of $ G $ on $ L^2(\Gamma\backslash G) $, where $ (\pi_\rho^*,\calH_\rho^*) $ is the contragredient representation of $ \pi_\rho $ (see Theorem \ref{thm:019}). Next, applying D.\ Mili\v ci\'c's result \cite[Lemma 6.6]{muic19} (see Proposition \ref{prop:027}), we prove that $ \calL_\rho^\infty(\Gamma) $ is the space of the (absolutely and locally uniformly convergent) Poincar\'e series
\[ \left(P_\Gamma c_{h,h'}\right)(g)=\sum_{\gamma\in\Gamma}c_{h,h'}(\gamma g) \]
of the matrix coefficients $ c_{h,h'}=\scal{\pi_\rho^*(\spacedcdot)h}{h'}_{\calH_\rho^*}:G\to\C $, where $ h $ is a weight $ -\Lambda $ vector of $ \pi_\rho^* $, and $ h' $ goes over the $ K $-finite vectors in $ \calH_\rho^* $. Finally, by computing the matrix coefficients $ c_{h,h'} $ (see Proposition \ref{prop:009}) and transferring the resulting description of $ \calL_\rho^\infty(\Gamma) $ via $  \Phi_{\rho,\Gamma}^{-1} $, we obtain Theorem \ref{thm:100}.

Let $ \mathfrak a $ be a maximal abelian subspace of $ \mathfrak p $. We fix a choice of the set $ \Sigma^+ $ of positive restricted roots of $ \mathfrak g $ with respect to $ \mathfrak a $ and let
\[ \mathfrak a^+=\left\{X\in\mathfrak a:\lambda(X)>0\text{ for all }\lambda\in\Sigma^+\right\}. \]
We denote by $ m_\lambda $ the multiplicity of a restricted root $ \lambda\in\Sigma^+ $.

With Theorem \ref{thm:100} and the linear isomorphism $ \Phi_{\rho,\Gamma}:\calH_\rho^\infty(\Gamma)\to\calL_\rho^\infty(\Gamma) $ at hand, in Section \ref{sec:154} we apply a vector-valued version, given by Proposition \ref{prop:101}, of Mui\'c's integral non-vanishing criterion \cite[Theorem 4.1]{muic09} to the Poincar\'e series $ P_{\Gamma}F_f=F_{P_{\Gamma,\rho}f} $ with $ f\in\mathcal P(\mathcal D,V) $, thus proving the following non-vanishing result.

\begin{theorem}\label{thm:103a}
	Let $ f\in\mathcal P(\mathcal D,V)\setminus\left\{0\right\} $. 
	Suppose that there exists a Borel set $ S\subseteq\mathfrak a^+ $ with the following properties:
	\begin{enumerate}[label=\textup{(S\arabic*)}]
		\item $ K\exp(S)K\exp(-S)K\cap\Gamma=\left\{1_{G}\right\} $.
		\item We have
			\[ \begin{aligned}
				\int_S\int_K&\norm{F_f(k\exp(X))}_V\left(\prod_{\lambda\in\Sigma^+}(\sinh\lambda(X))^{m_\lambda}\right)\,dk\,dX \\
				&>\frac12\int_{\mathfrak a^+}\int_K\norm{F_f(k\exp(X))}_V\left(\prod_{\lambda\in\Sigma^+}(\sinh\lambda(X))^{m_\lambda}\right)\,dk\,dX.
			\end{aligned}   \]
	\end{enumerate}
	Then, we have
	\[ P_{\Gamma,\rho}f\in\calH_\rho^\infty(\Gamma)\setminus\left\{0\right\} \]
	 and 
	 \[ P_{\Gamma}v_\Lambda^*F_f\in \calL_\rho^\infty(\Gamma)\setminus\left\{0\right\}. \]
\end{theorem}

In Section \ref{sec:160}, we make the above results explicit in the case when $ G=\SU(p,q) $ and $ K=\mathrm S(\U(p)\times\U(q)) $ with $ p\geq q\geq1 $. In Theorem \ref{thm:112}\ref{thm:112:3}, we rephrase Theorem \ref{thm:103a} in the case when
\[ \Gamma=\Gamma_G(N)\coloneqq\SU(p,q)\cap\big(I_{p+q}+M_{p+q}\left(N\Z\left[\sqrt{-1}\,\right]\right)\big) \]
for some $ N\in\Z_{>0} $, rendering it easily applicable to families of Poincar\'e series $ P_{\Gamma_G(N),\rho}f $ with fixed $ \rho $ and $ f $ and varying $ N $. In Example \ref{ex:162}, we give an example application of our non-vanishing result to an infinite family of scalar-valued Poincar\'e series, using Wolfram Mathematica 14.1 \cite{mathematica} for numerical computations; the code can be found in \cite{zunar25}.

\section{Basic notation}

Let $ i\in\C $ be the imaginary unit. Given a Lie group $ G $ with Lie algebra $ \mathfrak g $, let $ \calU(\mathfrak g) $ denote the universal enveloping algebra of the complex Lie algebra $ \mathfrak g_\C=\mathfrak g\otimes_\R\C $. Let $\mathcal Z(\mathfrak g) $ denote the center of the algebra $ \mathcal U(\mathfrak g) $.

Given a complex topological vector space $ H $, we denote by $ H^* $ the space of continuous linear functionals $ H\to\C $. If $ H $ is a Hilbert space, we denote by $ \scal\spacedcdot\spacedcdot_H $ the inner product on $ H $ and write $ h^*=\scal\spacedcdot h_H\in H^* $ for every $ h\in H $. Given a unitary representation $ \pi=(\pi,H) $ of a Lie group $ G $ on $ H $, we denote by $ H_K $ the $ (\mathfrak g,K) $-module of smooth, $ K $-finite vectors in $ H $ and, given a subset $ S $ of $ \calU(\mathfrak g) $, by $ (H_K)^S $ the space of vectors $ h\in H_K $ such that $ \pi(S)h=0 $. For every irreducible unitary representation $ \sigma $ of $ G $, we denote by $ H_{[\sigma]} $ the $ \sigma $-isotypic component of $ \pi $, i.e., the closure in $ H $ of the sum of irreducible closed $ G $-invariant subspaces of $ H $ that are equivalent to $ \sigma $. The contragredient representation $ \left(\pi^*,H^*\right) $ of $ \pi $ is given by
\begin{equation}\label{eq:035}
	\pi^*(g)h^*=\left(\pi(g)h\right)^*,\quad\  g\in G,\ h\in H.
\end{equation}

Given a unimodular Lie group $ G $ with a fixed Haar measure $ dg $, for every discrete subgroup $ \Gamma $ of $ G $ we equip the quotient  $ \Gamma\backslash G $ with the $ G $-invariant Radon measure such that
\[ \int_{\Gamma\backslash G}\sum_{\gamma\in\Gamma}\varphi(\gamma g)\,dg=\int_G\varphi(g)\,dg,\quad\ \varphi\in C_c(G), \]
where $ C_c(G) $ denotes the space of compactly supported, continuous functions $ G\to\C $.
Given $ p\in\R_{\geq1} $ and a finite-dimensional complex Hilbert space $ V $, let $ L^p(\Gamma\backslash G,V) $ denote the Banach space  of (equivalence classes of) $ p $-integrable functions $ \Gamma\backslash G\to V $. We write $ L^p(\Gamma\backslash G)=L^p(\Gamma\backslash G,\C) $ and denote by $ R $ the right regular representation of $ G $ on the Hilbert space $ L^2(\Gamma\backslash G,V) $. Given a function $ \varphi:G\to V $, we define a function $ \check\varphi:G\to V $,
\begin{equation}\label{eq:026}
	\check\varphi(g)=\varphi\big(g^{-1}\big).
\end{equation}

Let $ g^\top $, $ g^* $, and $ \mathrm{tr}(g) $ denote, respectively, the transpose, the conjugate transpose, and the trace of a complex matrix $ g $, and let $ \norm g=\left(\mathrm{tr}(g^*g)\right)^{\frac12} $. Given $ p,q,r\in\Z_{>0} $ such that $ r\leq\min\left\{p,q\right\} $ and $ t\in\C^r $, let 
\begin{equation}\label{eq:163}
	 t_{p\times q}=\mathrm{diag}(t_1,\ldots,t_r,0,\ldots,0)\in M_{p,q}(\C). 
\end{equation}
Given a function $ f:\C\to\C $ and $ t\in\C^r $, let $ f(t)=(f(t_1),\ldots,f(t_r))\in\C^r $.

\section{Representation-theoretic preliminaries}

Until the end of this section, we fix a connected semisimple Lie group $ G $ with finite center, a maximal compact subgroup $ K $ of $ G $, and a discrete subgroup $ \Gamma $ of $ G $.

Given a finite-dimensional complex vector space $ V $, we say that a function $ \varphi:G\to V $ is:
\begin{enumerate}[label={(\arabic*)}]
	\item\label{enum:045:1} left $ K $-finite if $ \dim_\C\mathrm{span}_\C\left\{\varphi(k\spacedcdot):k\in K\right\}<\infty $
	\item right $ K $-finite if $ \dim_\C\mathrm{span}_\C\left\{\varphi(\spacedcdot k):k\in K\right\}<\infty $
	\item\label{enum:045:3} $ \mathcal Z(\mathfrak g) $-finite if $ \varphi $ is smooth and $ \dim_\C\mathcal Z(\mathfrak g)\varphi<\infty $.
\end{enumerate}
By a straightforward generalization of Harish-Chandra's result \cite[Lemma 77]{harish66} from the $ L^2(G) $ case to the $ L^2(\Gamma\backslash G) $ case, we have the following lemma.

\begin{lemma}\label{lem:020}
	If a smooth function $ \varphi\in L^2(\Gamma\backslash G) $ is $ \mathcal Z(\mathfrak g) $-finite and right $ K $-finite, then the smallest closed $ G $-invariant subspace of $ (R,L^2(\Gamma\backslash G)) $ containing $ \varphi $ is an orthogonal sum of finitely many irreducible closed $ G $-invariant subspaces.
\end{lemma}

Given a unitary representation $ (\pi,H) $ of $ G $ and vectors $ h,h'\in H $, let $ c_{h,h'} $ denote the matrix coefficient $ \scal{\pi(\spacedcdot)h}{h'}_H:G\to\C $. An irreducible unitary representation $ (\pi,H) $ of $ G $ is said to be square-integrable (resp., integrable) if $ c_{h,h'}\in L^2(G) $ for all $ h,h'\in H $ (resp., if $ c_{h,h'}\in L^1(G) $ for all $ h,h'\in H_K $). 
If $ (\pi,H) $ is an integrable representation of $ G $, then for all $ h,h'\in H_K $, the matrix coefficient $ c_{h,h'} $ is a smooth function in $ L^1(G) $ with properties \ref{enum:045:1}--\ref{enum:045:3}. Thus, by the proof of \cite[Theorem 5.4]{baily_borel66} and the discussion in \cite[the beginning of \S6]{muic19}, we have the following lemma.

\begin{lemma}\label{lem:028}
	Let $ (\pi,H) $ be an integrable representation of $ G $, and let $ h,h'\in H_K $. Then, the Poincar\'e series 
	\[ \left(P_\Gamma c_{h,h'}\right)(g)=\sum_{\gamma\in\Gamma}c_{h,h'}(\gamma g) \]
	converges absolutely and uniformly on compact subsets of $ G $ and defines a bounded, right $ K $-finite, $ \mathcal Z(\mathfrak g) $-finite function $ \Gamma\backslash G\to\C $ that belongs to $ L^p(\Gamma\backslash G) $ for every $ p\in\R_{\geq1} $. 
\end{lemma}

We say that the multiplicity of an irreducible unitary representation $ \pi $ of $ G $ in $ \left(R,L^2(\Gamma\backslash G)\right) $ is finite, and write $ m_\Gamma(\pi)<\infty $, if $ L^2(\Gamma\backslash G)_{[\pi]} $ is an orthogonal sum of finitely many (possibly zero) irreducible closed $ G $-invariant subspaces. 

\begin{remark}\label{rem:152} 
	It is well known (see, e.g., \cite[Introduction]{harish68}) that $ m_\Gamma(\pi)<\infty $ for every $ \pi $ provided that $ G $ and $ \Gamma $ satisfy one of the following:
	\begin{enumerate}
		\item The quotient space $ \Gamma\backslash G $ is compact.
		\item $ G=\mathcal G(\mathbb R) $ for a connected semisimple algebraic group $ \mathcal G $ defined over $ \Q $, and $ \Gamma $ is an arithmetic subgroup of $ \mathcal G(\Q) $.
	\end{enumerate}
\end{remark}

The following proposition is a variant of Mili\v ci\'c's result \cite[Lemma 6.6]{muic19}. 

\begin{proposition}\label{prop:027}
	Let $ (\pi,H) $ be an integrable representation of $ G $ such that $ m_\Gamma(\pi)<\infty $. Then, we have
	\[ \left(L^2(\Gamma\backslash G)_{[\pi]}\right)_K=\mathrm{span}_\C\left\{P_\Gamma c_{h,h'}:h,h'\in H_K\right\}. \]
\end{proposition}

\begin{proof}
	By \cite[Lemmas 6.6 and 5.5(ii)(a) and Theorem 6.4(i)]{muic19}, the $ (\mathfrak g,K) $-submodule  
	\[ V_\pi:=\mathrm{span}_\C\left\{P_\Gamma c_{h,h'}:h,h'\in H_K\right\} \]
	of $ \left(L^2(\Gamma\backslash G)_{[\pi]}\right)_K $ is dense in $  L^2(\Gamma\backslash G)_{[\pi]} $. Thus, since the assumption that $ m_\Gamma(\pi)<\infty $ implies that $ L^2(\Gamma\backslash G)_{[\pi]} $ is an admissible representation of $ G $ (see \cite[Theorems 0.2 and 0.3]{knapp_vogan95}), it follows by  \cite[Theorem 0.4]{knapp_vogan95} that $ V_\pi=\left(L^2(\Gamma\backslash G)_{[\pi]}\right)_K $.
\end{proof}

\section{Harish-Chandra's realization of irreducible bounded symmetric domains}\label{sec:104}

From now until the end of Section \ref{sec:154}, we fix a simple, non-compact real Lie algebra $ \mathfrak g $ and its (up to conjugation by an inner automorphism of $ \mathfrak g $ unique) Cartan decomposition $ \mathfrak g=\mathfrak k\oplus\mathfrak p $ into a Lie subalgebra $ \mathfrak k $ and a subspace $ \mathfrak p $ such that $ \mathfrak k\oplus i\mathfrak p $ is a compact real form of $ \mathfrak g_\C $ (cf.\ \cite[Ch.\ III, \S7]{helgason78}). Moreover, we fix a maximal abelian subalgebra $ \mathfrak h $ of $ \mathfrak k $ and assume that $ \mathfrak h $ is also a maximal abelian subalgebra of $ \mathfrak g $. 
Thus, $ \mathfrak h_\C $ is a Cartan subalgebra of $ \mathfrak g_\C $, and the root system $ \Delta=\Delta(\mathfrak h_\C:\mathfrak g_\C) $ is a disjoint union of the set $ \Delta_K=\Delta(\mathfrak h_\C:\mathfrak k_\C) $ of compact roots and the set $ \Delta_n $ of non-compact roots (cf.\ \cite[\S3]{harish55_iv}). 

We fix a choice of the set $ \Delta^+ $ of positive roots in $ \Delta $, write $ \Delta_K^+=\Delta_K\cap\Delta^+ $ and $ \Delta_n^+=\Delta_n\cap\Delta^+ $, and denote by $ \mathfrak g_\alpha $ the root subspace corresponding to a root $ \alpha\in\Delta $. We suppose that every root $ \alpha\in\Delta_n^+ $ is totally positive, i.e., every root in $ \alpha+\mathrm{span}_\Z\Delta_K $ is positive (cf.\ \cite[Corollary of Lemma 12]{harish55_iv}). Thus, $ \mathfrak p_\C^+=\sum_{\alpha\in\Delta_n^+}\mathfrak g_\alpha $ and $ \mathfrak p_\C^-=\sum_{\alpha\in\Delta_n^+}\mathfrak g_{-\alpha} $ are abelian subspaces of $ \mathfrak g_\C $, and we have $ \left[\mathfrak k_\C,\mathfrak p_\C^+\right]\subseteq\mathfrak p_\C^+ $ and $ \left[\mathfrak k_\C,\mathfrak p_\C^-\right]\subseteq\mathfrak p_\C^- $ \cite[Lemmas 11 and 12]{harish55_iv}. 
Let
$ \mathfrak k_\C^+=\sum_{\alpha\in\Delta_K^+}\mathfrak g_\alpha $, $ \mathfrak k_\C^-=\sum_{\alpha\in\Delta_K^+}\mathfrak g_{-\alpha} $, $ \mathfrak n^+=\mathfrak k_\C^+\oplus\mathfrak p_\C^+ $, $ \mathfrak n^-=\mathfrak k_\C^-\oplus\mathfrak p_\C^- $, and $ \mathfrak b=\mathfrak k_\C\oplus\mathfrak p_\C^+ $.
Given $ \alpha\in\Delta $, let $ H_\alpha $ be the element of the one-dimensional subspace $ [\mathfrak g_\alpha,\mathfrak g_{-\alpha}]\subseteq \mathfrak h_\C $ such that $ \alpha(H_\alpha)=2 $. Let $ \delta=\frac12\sum_{\alpha\in\Delta^+}\alpha\in\mathfrak h_\C^* $.

Let $ G_\C $ be the simply connected complex Lie group with Lie algebra $ \mathfrak g_\C $, and let $ G $, $ K $, and $ T $ be, respectively, the connected (real) Lie subgroups of $ G_\C $ with Lie algebras $ \mathfrak g $, $ \mathfrak k $, and $ \mathfrak h $. The group $ G $ is a connected semisimple Lie group with finite center, and $ K $ is a maximal compact subgroup of $ G $ (cf.\ \cite[Ch.\,VI, Theorem 1.1(i)]{helgason78} and \cite[\S2]{harish56_v}). Let $ K_\C $, $ P_\C^+ $, and $ P_\C^- $ be, respectively, the connected complex Lie subgroups of $ G_\C $ with Lie algebras $ \mathfrak k_\C $, $ \mathfrak p_\C^+ $, and $ \mathfrak p_\C^- $. The group $ K_\C $ normalizes $ P_\C^+ $ and $ P_\C^- $, and the assignment $ (p^+,k,p^-)\mapsto p^+kp^- $ defines a biholomorphism from $ P_\C^+\times K_\C\times P_\C^- $ onto the open subset $ P_\C^+K_\C P_\C^- $ of $ G_\C $ (cf.\ \cite[Lemma 4]{harish56_v} or \cite[Lemma 2]{bruhat58}). In particular, given $ g\in P_\C^+K_\C P_\C^- $, we have 
	\begin{equation}\label{eq:200}
		 g=g_+g_0g_- 
	\end{equation}
for a unique choice of $ g_+\in P_\C^+ $, $ g_0\in K_\C $, and $ g_-\in P_\C^- $.

The set $ GK_\C P_\C^- $ is an open subset of $ P_\C^+K_\C P_\C^- $, and since $ G\cap K_\C P_\C^-=K $, we can identify the quotient manifold $ GK_\C P_\C^-/K_\C P_\C^-$ canonically with $ G/K $, thus endowing the smooth manifold $ G/K $ with a complex structure with respect to which the canonical left action of $ G $ on $ G/K $ is by byholomorphisms (cf.\ \cite[Lemma 3 and \S2]{bruhat58}). The exponential map $ \exp:\mathfrak p_\C^+\to P_\C^+ $ is biholomorphic, and the assignment
\[ gK\equiv gK_\C P_\C^-\mapsto\exp^{-1}(g_+),\quad\  g\in G, \]
defines a biholomorphism $ \zeta $ from $ G/K $ onto a bounded symmetric domain $ \mathcal D\subseteq\mathfrak p_\C^+ $ (cf.\ \cite[Theorem 1]{bruhat58}).
The canonical left action of $ G $ on $ G/K $ corresponds via $ \zeta $ to the following left action of $ G $ on $ \mathcal D $:
\[ g.x=\exp^{-1}\left(\left(g\exp x \right)_+\right),\quad\  g\in G,\ x\in \mathcal D. \]
Given $ x\in \mathcal D $, let $ \textbf{\textit x}=\zeta^{-1}(x) $. Thus, $ \textbf{\textit x}  $ is the unique element of $ G/K $ such that $ \textbf{\textit x}.0=x $. Since $ K $ normalizes $ P_\C^- $, we can use \eqref{eq:200} to define $ \textbf{\textit x}_0\in K_\C/K $.

Let us fix an (up to a positive multiplicative constant unique) non-zero $ G $-invariant Radon measure $ \mathsf v $ on $ \mathcal D $. Applying \cite[Theorem 8.36]{knapp02}, we specify a Haar measure $ dg $ on $ G $ by the condition
	\[  \int_G\varphi(g)\,dg=\int_{\mathcal D}\int_K\varphi(\textbf{\textit x}k)\,dk\,d\mathsf v(x),\quad\  \varphi\in C_c(G),  \]
where $ dk $ is the Haar measure on $ K $ such that $ \int_Kdk=1 $.

\section{Holomorphic and anti-holomorphic discrete series of $ G $}

Let $ (\pi,H) $ be a unitary representation of $ G $ (resp., $ K $). A linear functional $ \lambda\in\mathfrak h_\C^* $ is a weight of $ \pi $ if the subspace
\[ H_{\lambda}=\left\{h\in H_K:\pi(X)h=\lambda(X)h\text{ for all }X\in\mathfrak h_\C\right\} \]
of $ H_K $ is non-trivial. If $ \pi $ is irreducible, then there exists at most one weight $ \lambda $ of $ \pi $ such that $ \pi(\mathfrak n^+)v=0 $ (resp., $ \pi(\mathfrak k_\C^+)v=0 $) for some $ v\in H_\lambda\setminus\left\{0\right\} $ (see \cite[\S20.2]{humphreys72}), in which case $ \pi $ is said to be of highest weight $ \lambda $ (with respect to the positive system $ \Delta^+ $ (resp., $ \Delta_K^+ $)), and the non-zero elements of the (necessarily one-dimensional) subspace $ H_\lambda $ are called the highest weight vectors of $ \pi $.
Irreducible unitary representations of $ G $ (resp., $ K $) of lowest weight $ \lambda $ and their lowest weight vectors are defined by replacing $ \mathfrak n^+ $ (resp., $ \mathfrak k_\C^+ $) in the above definition by $ \mathfrak n^- $ (resp., $ \mathfrak k_\C^- $).

From now until the end of Section \ref{sec:154}, we fix $ \Lambda\in\mathfrak h_\C^* $ with the following properties:
\begin{enumerate}[label={($ \Lambda $\arabic*)}]
	\item\label{enum:002:1} $ \Lambda $ is analytically integral, i.e., $ \Lambda\left(\left\{X\in\mathfrak h:\exp(X)=1_T\right\}\right)\subseteq 2\pi i \Z $.
	\item\label{enum:002:2} $ \Lambda(H_\alpha)\geq0 $ for all $ \alpha\in\Delta_K^+ $.
	\item\label{enum:002:3} $ (\Lambda+\delta)(H_\alpha)<0 $ for all $ \alpha\in\Delta_n^+ $.
\end{enumerate}

\begin{lemma}\label{lem:022}
	There exist up to equivalence unique irreducible unitary representations $ (\rho,V) $ of $ K $ and $ (\pi_\rho,\calH_\rho) $ of $ G $ that are of highest weight $ \Lambda $. The representations $ (\rho,V) $ and $ (\pi_\rho,\calH_\rho) $ and their contragredient representations $ (\rho^*,V^*) $ and $ \left(\pi_\rho^*,\calH_\rho^*\right) $ have the following properties:
	\begin{enumerate}[label={(\roman*)}]
		\item\label{lem:022:1} The representation $ \left(\pi_\rho^*,\calH_\rho^*\right) $ (resp., $ (\rho^*,V^*) $) is the up to equivalence unique irreducible unitary representation of $ G $ (resp., of $ K $) of lowest weight $ -\Lambda $, and we have
		\begin{equation}\label{eq:102}
			\left(\calH_\rho^*\right)_{-\Lambda}=\left\{h^*:h\in\left(\calH_\rho\right)_\Lambda\right\}\quad\ \text{and}\quad\ \left(V^*\right)_{-\Lambda}=\left\{v^*:v\in V_\Lambda\right\}.
		\end{equation}
		\item\label{lem:022:2a} The $ (\mathfrak g,K) $-module $ \left(\calH_\rho\right)_K $ (resp., $ \left(\calH_\rho^*\right)_K $) decomposes into a direct sum of its weight subspaces.
		\item\label{lem:022:2}  As a left $ \mathcal U(\mathfrak g) $-module, $ \left(\calH_\rho\right)_K $ is isomorphic to the generalized Verma module
		\begin{equation}\label{eq:105}
			\mathcal U(\mathfrak g)\otimes_{\mathcal U(\mathfrak b)}V,
		\end{equation}
		where $ V $ is regarded as the left $ \mathcal U(\mathfrak b) $-module obtained from the $ \mathcal U(\mathfrak k) $-module $ (\rho,V) $ by letting $ \mathfrak p_\C^+ $ act trivially.
		\item\label{lem:022:3} In terms of \cite[Theorem 9.20]{knapp86}, $ \pi_\rho $ (resp., $ \pi_\rho^* $) is a discrete series representation of $ G $ with Harish-Chandra parameter $ \Lambda+\delta $ (resp., $ -\Lambda-\delta $) and Blattner parameter $ \Lambda $ (resp., $ -\Lambda $). 
		\item\label{lem:022:4} The $ K $-type $ \rho $ (resp., $ \rho^* $) occurs with multiplicity one in $ \pi_\rho $ (resp., $ \pi_\rho^* $), and we have 
		\begin{equation}\label{eq:041}
			\pi_\rho(\mathfrak p_\C^+)\,(\calH_\rho)_{[\rho]}=0\qquad\text{and}\qquad \pi_\rho^*(\mathfrak p_\C^-)\,(\calH_\rho^*)_{[\rho^*]}=0.
		\end{equation}
		\item\label{lem:022:5} We have the following equalities of one-dimensional complex vector spaces:
		\begin{enumerate}[label={(vi.\roman*)}]
			\item\label{lem:022:5:1} $ \left(\calH_\rho\right)_\Lambda=\left(\left(\calH_\rho\right)_{[\rho]}\right)_\Lambda=\big((\calH_\rho)_{[\rho]}\big)^{\mathfrak k_\C^+}=\big((\calH_\rho)_K\big)^{\mathfrak n^+}  $.
			\item\label{lem:022:5:2} $ \left(\calH_\rho^*\right)_{-\Lambda}=\left(\left(\calH_\rho^*\right)_{[\rho^*]}\right)_{-\Lambda}=\left((\calH_\rho^*)_{[\rho^*]}\right)^{\mathfrak k_\C^-}=\left((\calH_\rho^*)_K\right)^{\mathfrak n^-}. $
		\end{enumerate}
		\item\label{lem:022:6} The representation $ \pi_\rho $ (resp., $ \pi_\rho^* $) is integrable if and only if we have
			\begin{equation}\label{eq:110}
				(\Lambda+\delta)(H_\alpha)<-\frac14\sum_{\beta\in\Delta}\abs{\beta(H_\alpha)},\quad\  \alpha\in\Delta_n^+.
			\end{equation}
	\end{enumerate}
\end{lemma}

\begin{proof}
	The existence and uniqueness of the representation $ (\rho,V) $ follow from \ref{enum:002:1} and \ref{enum:002:2} by \cite[Theorems 4.28 and 1.15]{knapp86}. The representation $ \pi_\rho $ was constructed and investigated by Harish-Chandra in \cite{harish55_iv}, \cite{harish56_v}, and \cite{harish56_vi}. Its uniqueness follows from \cite[Theorem 2]{harish55_iv} and \cite[Theorem 8]{harish53}. 
	
	Let $ \eta_{\mathfrak g}:\mathfrak g_\C\to\mathfrak g_\C $ be the conjugation of $ \mathfrak g_\C $ with respect to $ \mathfrak g $, given by 
	\[ \eta_{\mathfrak g}(X+iY)=X-iY,\quad\  X,Y\in\mathfrak g.  \]
	Since $ \Lambda(\mathfrak h)\subseteq i\mathbb R $ and $ \eta_{\mathfrak g}(\mathfrak g_\alpha)=\mathfrak g_{-\alpha} $ for all $ \alpha\in\Delta $ \cite[p.\,757]{harish55_iv}, the elementary equality
	\begin{equation}\label{eq:100}
		\pi_\rho^*(X)h^*=\big(\pi_\rho(\eta_{\mathfrak g}(X))h\big)^*,\quad\  X\in\mathfrak g_\C,\ h\in\left(\calH_\rho\right)_K,
	\end{equation}
	implies that if $ h $ is the highest weight vector of $ \pi_\rho $, then the linear functional $ h^* $ is the lowest weight vector, of weight $ -\Lambda $, of $ \pi_\rho^* $. This proves the claim on $ \pi_\rho^* $ in \ref{lem:022:1}; the claim on $ \rho^* $ is proved analogously.
	
	The claim \ref{lem:022:2a} follows from \ref{lem:022:1} by \cite[\S20.2, Theorem (a)]{humphreys72}. The claims \ref{lem:022:2} and \ref{lem:022:3} are well known (cf.\ \cite{repka79} and \cite[Remark (1) after Theorem 9.20]{knapp86}). By \cite[Theorem 9.20(b)]{knapp86} and \cite[Theorem (8.5)]{atiyah_schmid77}, \ref{lem:022:3} implies the multiplicity one claim in \ref{lem:022:4}. The first equality in \ref{lem:022:5:1} follows immediately since $ \dim_\C(\calH_\rho)_\Lambda=1 $ by \cite[\S20.2, Theorem (c)]{humphreys72}. The remaining equalities in \ref{lem:022:5:1} follow from \cite[\S20.2, Corollary]{humphreys72}, and \ref{lem:022:5:2} follows by duality. The first equality in \eqref{eq:041} is a restatement of the fact that in \eqref{eq:105}, $ \mathfrak p_\C^+ $ annihilates $\C\otimes V $. The second equality in \eqref{eq:041} follows by duality as in the proof of \ref{lem:022:1}. 
	
	Finally, \ref{lem:022:6} follows from \ref{lem:022:3} and \ref{enum:002:3} by applying the criterion \cite{hecht_schmid76}, or equivalently \cite[Part III, \S3, Cor.\ of Theorem 1]{milicic77}, for the integrability of a discrete series representation in terms of its Harish-Chandra parameter.
\end{proof}

The representations $ \pi_\rho $ (resp., $ \pi_\rho^* $) of Lemma \ref{lem:022} constitute the holomorphic (resp., anti-holomorphic) discrete series of $ G $.
From now until the end of Section \ref{sec:154}, we fix a representation $ (\rho,V) $ of $ K $ as in Lemma \ref{lem:022} and a highest weight vector $ v_\Lambda $ of $ \rho $ such that $ \norm{v_\Lambda}_V=1 $. The unique extension of $ \rho $ to an irreducible holomorphic representation of $ K_\C $ will also be denoted by $ \rho $. 

Following \cite{bruhat58}, we will work with the realization of the representation $ (\pi_\rho,\calH_\rho) $ in which $ \calH_\rho $ is the Hilbert space of holomorphic functions $ \varphi:P_\C^-K_\C G\to V $ such that
\begin{equation}\label{eq:003}
	\varphi\left(p^-kw\right)=\rho(k)\varphi(w),\quad\ p^-\in P_\C^-,\ k\in K_\C,\ w\in P_\C^-K_\C G, 
\end{equation}
and
\begin{equation}\label{eq:107}
	\int_G\norm{\varphi(g)}^2_{V}\,dg<\infty;
\end{equation}
the norm on $ \calH_\rho $ is given by the square root of \eqref{eq:107}, and the action of $ G $ is given by
\[ \pi_\rho(g)\varphi=\varphi\left(\spacedcdot g\right),\quad\ g\in G,\ \varphi\in\calH_\rho. \]

Restricting the functions in $ \calH_\rho $ to $ G $ yields a unitary equivalence $ \Psi_\rho $ of $ \left(\pi_\rho,\calH_\rho\right) $ with an irreducible closed $ G $-invariant subspace $ \mathcal R_\rho $ of the right regular representation $ \left(R,L^2(G,V)\right) $. By \eqref{eq:003}, for every $ F\in\calR_\rho $, we have
\begin{equation}\label{eq:013}
	F(kg)=\rho(k)F(g),\quad\  k\in K,\ g\in G.
\end{equation}

Given $ \varphi\in\calH_\rho $, we define a function $ \widetilde\varphi:\mathcal D\to V $,
\[ \widetilde\varphi(x)=\varphi\left((\exp x)^{-1}\right). \]
By a computation analogous to that in \cite[\S4]{repka79}, one shows that the assignment $ \varphi\mapsto \widetilde\varphi $ defines a unitary $ G $-equivalence $ \widetilde\Psi_\rho $ of $ \pi_\rho $ with a representation $ \widetilde\pi_\rho $ of $ G $ on the Hilbert space $ \widetilde\calH_\rho $ of holomorphic functions $ f:\mathcal D\to V $ such that
\begin{equation}\label{eq:108}
	\int_{\mathcal D}\norm{\rho(\textbf{\textit x}_{0})^{-1}f(x)}_V^2\,d\mathsf v(x)<\infty;
\end{equation}
the norm on $ \widetilde\calH_\rho $ is given by the square root of \eqref{eq:108}, and the action of $ G $ is given by
\[ \widetilde\pi_\rho(g)f=f\big|_\rho g^{-1},\quad\  g\in G,\ f\in\widetilde\calH_\rho, \]
where $ \big|_\rho $ denotes the right action of $ G $ on the space $ V^{\mathcal D} $ of functions $ \mathcal D\to V $ given by
\begin{equation}\label{eq:030}
	\left(f\big|_\rho g\right)(x)=\rho\big((g\exp x)_0\big)^{-1}f(g.x),\quad\  f\in V^{\mathcal D},\ g\in G,\ x\in\mathcal D.
\end{equation}
 
The $ (\mathfrak g,K) $-module $ \big(\widetilde\calH_\rho\big)_K $ is the space $ \mathcal P(\mathcal D,V) $ of polynomial functions from the bounded domain $ \mathcal D\subseteq\mathfrak p_\C^+ $ to $ V $ (cf.\ \cite[Remarque 1]{bruhat58}). Let $ \mathcal P(\mathcal D) $ denote the space of polynomial functions $ \mathcal D\to\C $. The assignment 
\[ \mu\otimes v\mapsto f_{\mu\otimes v}:=\mu(\spacedcdot)v \]
defines a linear isomorphism $ \mathcal P(\mathcal D)\otimes_\C V\to\mathcal P(\mathcal D,V) $. The assignment $ v\mapsto f_{1\otimes v} $ defines a $ K $-equivalence of $ (\rho,V) $ with the space $ \big(\widetilde\calH_\rho\big)_{[\rho]} $ of constant functions $ \mathcal D\to V $.

Given $ f:\mathcal D\to V $, we define a function $ F_f:G\to V $,
\begin{equation}\label{eq:032}
	F_f(g)=\left(f\big|_\rho g\right)(0)=\rho(g_0)^{-1}f(g.0).
\end{equation}
We will write 
\[ F_{\mu\otimes v}=F_{f_{\mu\otimes v}}\qquad\text{and}\qquad F_v=F_{1\otimes v} \]
for $ \mu\in\mathcal P(\mathcal D) $ and $ v\in V $. The following lemma is elementary.
\begin{lemma}\label{lem:036}
	The assignment $ f\mapsto F_f $ defines a linear isomorphism from $ V^{\mathcal D} $ to the space of functions $ F:G\to V $ such that
	\begin{equation}\label{eq:017}
		F(gk)=\rho(k)^{-1}F(g),\quad\  g\in G,\ k\in K. 
	\end{equation}
	In particular, for every $ f\in V^{\mathcal D} $, we have
	\begin{equation}\label{eq:125}
		\norm{F_f(gk)}_V=\norm{F_f(g)}_V,\quad\  g\in G,\ k\in K. 
	\end{equation}
\end{lemma}
We note that the unitary $ G $-equivalence $ \Psi_\rho\circ\widetilde\Psi_\rho^{-1}:\widetilde\calH_\rho\to\calR_\rho $ is given by the assignment
\begin{equation}\label{eq:109}
	f\mapsto \check F_f,
\end{equation}
where we use the notation \eqref{eq:026}.
In particular, we have
\begin{equation}\label{eq:043}
	(\calR_\rho)_K=\left\{\check F_f:f\in \mathcal P(\mathcal D,V)\right\},
\end{equation}
and a $ K $-equivalence from $ (\rho,V) $ to $ \left(\calR_\rho\right)_{[\rho]} $ is given by the assignment
\begin{equation}\label{eq:007}
	v\mapsto \check F_{v}.
\end{equation}
By Lemma \ref{lem:022}\ref{lem:022:5:1} and \eqref{eq:102}, it follows that 
\begin{equation}\label{eq:042}
	(\calR_\rho)_\Lambda=\C\check F_{v_\Lambda}\qquad\text{and}\qquad\big(\calR_\rho^*\big)_{-\Lambda}=\C\big(\check F_{v_\Lambda}\big)^*.
\end{equation}

\section{Determination of matrix coefficients}

Given $ F:G\to V $ and $ v\in V $, we define a function $ v^*F:G\to\C $, 
\begin{equation}\label{eq:024}
	 \left(v^*F\right)(g)=\scal{F(g)}v_V. 
\end{equation}

\begin{lemma}\label{lem:012}
	Let $ v\in V\setminus\left\{0\right\} $. Let $ F:G\to V $ be such that \eqref{eq:013} or \eqref{eq:017} holds. If $ v^*F=0 $, then $ F=0 $.
\end{lemma}

\begin{proof}
	Suppose that $ v^*F=0 $ and \eqref{eq:013} holds. Then, for every $ k\in K $, we have
	\[ \scal{F(\spacedcdot)}{\rho(k)v}_V=\scal{\rho(k)^{-1}F(\spacedcdot)}v_V\overset{\eqref{eq:013}}=\scal{F\big(k^{-1}\spacedcdot\big)}v_V=\left(v^*F\right)\big(k^{-1}\spacedcdot\big)=0. \]
	Since $ \mathrm{span}_\C\rho(K)v=V $, this implies that $ F=0 $. The proof in the case when $ F $ satisfies \eqref{eq:017} instead of \eqref{eq:013} is analogous.
\end{proof}

Let us define a positive constant
\[ C_\rho=\norm{\check F_{v_\Lambda}}_{\calR_\rho}^2\overset{\eqref{eq:109}}=\norm{f_{1\otimes v_\Lambda}}^2_{\widetilde\calH_\rho}=\int_{\mathcal D}\norm{\rho(\textbf{\textit x}_0)^{-1}v_\Lambda}_V^2\,d\mathsf v(x). \]
The following proposition gives a formula for some $ K $-finite matrix coefficients of the representations $ \left(\pi_\rho,\calH_\rho\right)\cong\calR_\rho $ and $ \left(\pi_\rho^*,\calH_\rho^*\right)\cong\calR_\rho^* $.

\begin{proposition}\label{prop:009}
	Let $ F\in\left(\calR_\rho\right)_K $ and $ v\in V $. Then, using the notation introduced in \eqref{eq:026} and \eqref{eq:024}, we have the following:
	\begin{enumerate}[label={(\roman*)}]
		\item\label{prop:009:1} The matrix coefficient $ c_{F,\check F_v} $ of $ \calR_\rho $ is given by
			\[ c_{F,\check F_v}=C_\rho\, v^*F. \]
		\item\label{prop:009:2} The matrix coefficient $ c_{(\check F_v)^*,F^*} $ of $ \calR_\rho^* $ is given by
			\[ c_{(\check F_v)^*,F^*}=C_\rho\, v^*\check F.  \]
	\end{enumerate}
\end{proposition}

\begin{proof}
	We adapt the method of proof of \cite[Lemma 3-5]{muic10} and \cite[Proposition 5.3]{zunar23}.
	
	Let $ F\in\left(\calR_\rho\right)_K\setminus\left\{0\right\} $. Then, $ F $ is left $ K $-finite by \eqref{eq:013} and $ \mathcal Z(\mathfrak g) $-finite, hence by \cite[Theorem 1]{harish66} there exists a smooth function $ \alpha\in C_c(G) $ such that $ F=\check\alpha*F $. Thus, we have
	\[ \begin{aligned}
		\left(v^*F\right)(g)
		&=\scal{\int_G\check\alpha\big(h^{-1}\big)F(hg)\,dh}v_V\\
		&=\int_G\scal{\alpha(h)\,F(hg)}v_V\,dh\\
		&=\int_G\scal{F(hg)}{\overline{\alpha(h)}\,v}_V\,dh\\
		&=\scal{R(g)F}{\overline\alpha v}_{L^2(G,V)}\\
		&=\scal{R(g)F}{\mathrm{pr}_{\calR_\rho}(\overline\alpha v)}_{\calR_\rho},\quad\  g\in G,\ v\in V,
	\end{aligned} \]
	where $ \mathrm{pr}_{\calR_\rho} $ is the orthogonal projection $ L^2(G,V)\twoheadrightarrow\calR_\rho $. This shows that we have
	\begin{equation}\label{eq:004}
		v^*F=c_{F,\Phi_F(v)},\quad\  v\in V,
	\end{equation}
	for some function $ \Phi_F:V\to\calR_\rho $, which is easily shown to be a $ K $-equivariant linear operator: e.g., to prove its $ K $-equivariance, it suffices to note that for all $ k\in K $ and $ v\in V $, we have
	\[ \begin{aligned}
		c_{F,\Phi_F(\rho(k)v)}&\overset{\eqref{eq:004}}
		=\scal{F(\spacedcdot)}{\rho(k)v}_V
		=\scal{\rho(k)^{-1}F(\spacedcdot)}v_V
		\overset{\eqref{eq:013}}=\scal{F\left(k^{-1}\spacedcdot\right)}v_V\\
		&\overset{\phantom{\eqref{eq:004}}}=(v^*F)\left(k^{-1}\spacedcdot\right)
		\overset{\eqref{eq:004}}=c_{F,\Phi_F(v)}\left(k^{-1}\spacedcdot\right)
		=c_{F,R(k)\Phi_F(v)},
	\end{aligned} \]
	which implies that $ \Phi_F(\rho(k)v)=R(k)\Phi_F(v) $ by the definition of matrix coefficients and irreducibility of $ \calR_\rho $. 
	Since $ \rho $ occurs in $ \calR_\rho $ with multiplicity one (see Lemma \ref{lem:022}\ref{lem:022:4}), by Schur's lemma it follows that $ \Phi_F $ coincides up to a multiplicative constant with the $ K $-equivariant linear operator $ V\to\calR_\rho $ given by the assignment $ v\mapsto \check F_{v} $ (see \eqref{eq:007}). 
	
	Thus, there exists a function $ C:\left(\calR_\rho\right)_K\to\C $ such that
	\begin{equation}\label{eq:008}
		c_{F,\check F_v}=C(F)\,v^*F,\quad\  F\in\left(\calR_\rho\right)_K,\ v\in V. 
	\end{equation}
	Since for any fixed vector $ v\in V\setminus\left\{0\right\} $, the assignment
	\[ F\mapsto c_{F,\check F_v}=v^*\big(C(F)F\big)\mapsto C(F)F \]
	by Lemma \ref{lem:012} defines a linear operator on $ \left(\calR_\rho\right)_K $, the function $ C $ is necessarily constant on $ \left(\calR_\rho\right)_K\setminus\left\{0\right\} $. To determine its value, we evaluate \eqref{eq:008} at $ 1_G $ in the case when $ F=\check F_{v_\Lambda} $ and $ v=v_\Lambda $ and obtain
	\[ C(F)=\frac{c_{\check F_{v_\Lambda},\check F_{v_\Lambda}}(1_G)}{\left(v_\Lambda^*\check F_{v_\Lambda}\right)(1_G)}=\frac{\norm{\check F_{v_\Lambda}}_{\calR_\rho}^2}{\norm {v_\Lambda}_V^2}=\frac{C_\rho}1=C_\rho \]
	for all $ F\in \left(\calR_\rho\right)_K\setminus\left\{0\right\} $. Thus, again by \eqref{eq:008}, we have
		\[ c_{F,\check F_v}=C_\rho\, v^*F,\quad\  F\in\left(\calR_\rho\right)_K,\ v\in V.  \]
	This proves \ref{prop:009:1}, and \ref{prop:009:2} follows immediately by the elementary equality
	\[ c_{h^*,(h')^*}=\check c_{h',h}, \quad\  h,h'\in (\calR_\rho)_K.\qedhere \]
\end{proof}

\begin{corollary}\label{cor:124}
	Suppose that the representation $ \pi_\rho $ is integrable (see Lemma \ref{lem:022}\ref{lem:022:6}). Let $ f\in\mathcal P(\mathcal D,V) $. Then, $ F_f\in L^1(G,V) $.
\end{corollary}

\begin{proof}
	By \eqref{eq:043} and Proposition \ref{prop:009}\ref{prop:009:2}, for every $ v\in V $, $ v^*F_f $ is a $ K $-finite matrix coefficient of $ \mathcal R_\rho^*\cong(\pi_\rho^*,\calH_\rho^*) $ hence belongs to $ L^1(G) $ by Lemma \ref{lem:022}\ref{lem:022:6}. Thus, $ F_f\in L^1(G,V) $.
\end{proof}

\section{A space of holomorphic automorphic forms on $ \mathcal D $}\label{sec:102}

In this section, we suppose that the representation $ \pi_\rho $ is integrable (see Lemma \ref{lem:022}\ref{lem:022:6}). Let $ \Gamma $ be a discrete subgroup of finite covolume in $ G $ such that $ m_\Gamma(\pi_\rho)<\infty $. Let $ \calH_\rho^\infty(\Gamma) $ denote the space of holomorphic functions $ f:\mathcal D\to V $ with the following two properties:
\begin{enumerate}[label={(\arabic*)}]
	\item $ f\big|_\rho\gamma=f $ for all $ \gamma\in\Gamma $.
	\item $ \sup_{x\in \mathcal D}\norm{\rho(\textbf{\textit x}_0)^{-1}f(x)}_V<\infty $.
\end{enumerate}
The following theorem identifies $ \calH_\rho^\infty(\Gamma) $ with a finite-dimensional space of square-integrable automorphic forms on $ \Gamma\backslash G $. This result is well known to the experts, but we include a proof due to the lack of a suitable reference.

\begin{theorem}\label{thm:019}
	Suppose that the representation $ \pi_\rho $ is integrable (see Lemma \ref{lem:022}\ref{lem:022:6}).
	Let $ \Gamma $ be a discrete subgroup of finite covolume in $ G $ such that $ m_\Gamma(\pi_\rho)<\infty $. Then, the assignment $ f\mapsto v_\Lambda^*F_f $ defines a linear isomorphism
	\[ \Phi_{\rho,\Gamma}:\calH_\rho^\infty(\Gamma)\to \left(L^2(\Gamma\backslash G)_{[\pi_\rho^*]}\right)_{-\Lambda}.  \]
\end{theorem}

Our proof of Theorem \ref{thm:019} will rely on the properties of the lift $ f\mapsto F_f $ listed in the following lemma.

\begin{lemma}\label{lem:018}
	\begin{enumerate}[label={(\roman*)}]
		\item\label{lem:018:1} Given $ f:\mathcal D\to V $ and $ \gamma\in G $, the following equivalence holds:
		\[ f\big|_\rho\gamma=f\quad\Leftrightarrow\quad F_f(\gamma\spacedcdot)=F_f. \]
		\item\label{lem:018:2} For every $ f:\mathcal D\to V $, we have
		\[ \sup_{g\in G}\norm{F_f(g)}_V=\sup_{x\in \mathcal D}\norm{\rho(\textbf{\textit x}_{\,0})^{-1}f(x)}_V. \]
		\item\label{lem:018:3} A smooth function $ f:\mathcal D\to V $ is holomorphic if and only if $ \mathfrak p_\C^-F_f=0 $. Here, $ \mathfrak p_\C^-\subseteq\mathcal U(\mathfrak g) $ acts by left-invariant differential operators.
		\item\label{lem:018:4} Let $ F:G\to V $ be a smooth function that satisfies \eqref{eq:017} and is such that $ \mathfrak p_\C^- F=0 $. Then, $ F $ is $ \mathcal Z(\mathfrak g) $-finite.
		\item\label{lem:018:5} Let $ F\in L^2(\Gamma\backslash G,V)\setminus\left\{0\right\} $ be such that \eqref{eq:017} holds. Then, the assignment $ v^*\mapsto v^*F $ defines a $ K $-equivalence of $ (\rho^*,V^*) $ with the irreducible $ K $-invariant subspace $ V^*F:=\left\{v^*F:v\in V\right\} $ of $ (R,L^2(\Gamma\backslash G)) $. In particular, $ v_\Lambda^*F\in\left(V^*F\right)_{-\Lambda} $ is the lowest weight vector of the $ K $-representation $ V^*F $.
	\end{enumerate}
\end{lemma}

\begin{proof}
	The claims \ref{lem:018:1} and \ref{lem:018:2} are elementary. So is, by Lemma \ref{lem:012}, the first claim in \ref{lem:018:5}, which implies the second claim in \ref{lem:018:5} by Lemma \ref{lem:022}\ref{lem:022:1}. The claims \ref{lem:018:3} and \ref{lem:018:4} follow from \cite[Lemma 5.7]{baily_borel66}. 
\end{proof}

\begin{proof}[Proof of Theorem \ref{thm:019}]
	By Lemmmas \ref{lem:036} and \ref{lem:018}\ref{lem:018:1}--\ref{lem:018:3}, the assignment $ f\mapsto F_f $ defines a linear isomorphism from $ \calH_\rho^\infty(\Gamma) $ onto the space $ L_\rho(\Gamma) $ of bounded, smooth functions $ F:\Gamma\backslash G\to V $ that are annihilated by $ \mathfrak p_\C^- $ and satisfy \eqref{eq:017}. Thus, it remains to show that the assignment $ F\mapsto v_\Lambda^*F $ defines a linear isomorphism
	\[ \Theta:L_\rho(\Gamma)\to \left(L^2(\Gamma\backslash G)_{[\pi_\rho^*]}\right)_{-\Lambda}. \]
	
	Let $ F\in L_\rho(\Gamma) $. Then, $ v_\Lambda^*F\in L^2(\Gamma\backslash G) $ since $ F $ is bounded and $ \vol(\Gamma\backslash G)<\infty $. Moreover, the function $ v_\Lambda^*F $ is smooth, $ \mathcal Z(\mathfrak g) $-finite by Lemma \ref{lem:018}\ref{lem:018:4}, and right $ K $-finite by \eqref{eq:017}. Thus, by Lemma \ref{lem:020}, the smallest closed $ G $-invariant subspace $ H $ of $ \left(R,L^2(\Gamma\backslash G)\right) $ containing $ v_\Lambda^*F $ is an orthogonal sum of finitely many irreducible closed $ G $-invariant subspaces $ H_j $. Furthermore, $ v_\Lambda^*F $ belongs to $ H_{-\Lambda} $ by Lemma \ref{lem:018}\ref{lem:018:5} and is annihilated by both $ \mathfrak k_\C^- $ (by Lemma \ref{lem:018}\ref{lem:018:5}) and $ \mathfrak p_\C^- $ (since $ F $ is) and hence by $ \mathfrak n^- $. By the $ G $-equivariance of the orthogonal projections $ \mathrm{pr}_{H_j}:H\twoheadrightarrow H_j $, the same holds for the functions $ \mathrm{pr}_{H_j}v_\Lambda^*F\in (H_j)_K\setminus\left\{0\right\} $. Thus, by Lemma \ref{lem:022}\ref{lem:022:1}, the subspaces $ H_j $ are unitarily equivalent to $ \pi_\rho^* $, hence $ H\subseteq L^2(\Gamma\backslash G)_{[\pi_\rho^*]} $. Thus, we have 
	\[ v_\Lambda^*F\in H_{-\Lambda}\subseteq \left(L^2(\Gamma\backslash G)_{[\pi_\rho^*]}\right)_{-\Lambda}. \]
	This shows that the operator $ \Theta $ is well-defined.
	 
	The operator $ \Theta $ is injective by Lemma \ref{lem:012}. To prove its surjectivity, let $ \varphi\in\left(L^2(\Gamma\backslash G)_{[\pi_\rho^*]}\right)_{-\Lambda} $.  Since $ m_\Gamma(\pi_\rho^*)=m_\Gamma(\pi_\rho)<\infty $, we can assume without loss of generality that the smallest closed $ G $-invariant subspace $ H_\varphi $ of $ L^2(\Gamma\backslash G)_{[\pi_\rho^*]} $ containing $ \varphi $ is irreducible, hence equivalent to $ \pi_\rho^* $. Since $ \varphi\in(H_\varphi)_{-\Lambda}\setminus\left\{0\right\} $, Lemma \ref{lem:022}\ref{lem:022:5} implies that $ \varphi $ is the lowest weight vector of the irreducible $ K $-representation $ (H_\varphi)_{[\rho^*]} $, hence there exists a $ K $-equivalence $ \Phi:V^*\to(H_\varphi)_{[\rho^*]} $ such that $ \Phi(v_\Lambda^*)=\varphi $. Since by Proposition \ref{prop:027}, Lemma \ref{lem:028}, and \eqref{eq:041}, the elements of $ (H_\varphi)_{[\rho^*]}\subseteq \left(L^2(\Gamma\backslash G)_{[\pi_\rho^*]}\right)_K $ are smooth, bounded functions $ \Gamma\backslash G\to\C $ annihilated by $ \mathfrak p_\C^- $, we can define a function $ F:\Gamma\backslash G\to V $ with the same properties by the condition
	 \begin{equation}\label{eq:029}
	 	\scal{F(g)}v_V=\big(\Phi(v^*)\big)(g),\quad\  g\in G,\ v\in V. 
	 \end{equation}
	For all $ k\in K $ and $ v\in V $, we have
	\[ \begin{aligned}
		\scal{F(\spacedcdot k)}v_V
		&\overset{\eqref{eq:029}}
		=R(k)\Phi(v^*)
		=\Phi\rho^*(k)v^*
		=\Phi(\rho(k)v)^*\\
		&\overset{\eqref{eq:029}}=\scal{F(\spacedcdot)}{\rho(k)v}_V
		=\scal{\rho(k)^{-1}F(\spacedcdot)}{v}_V.
	\end{aligned} \]
	Thus, $ F $ satisfies \eqref{eq:017} hence belongs to $ L_\rho(\Gamma) $. Finally, we have
	\[ \Theta F=v_\Lambda^*F\overset{\eqref{eq:029}}=\Phi(v_\Lambda^*)=\varphi. \]
	This shows that the operator $ \Theta $ is surjective.
\end{proof}

Our next goal is to prove Theorem \ref{thm:100}. We need the following lemma.

\begin{lemma}\label{lem:038}
	Let $ \Gamma $ be a discrete subgroup of $ G $. Let $ f:\mathcal D\to V $ and $ v\in V\setminus\left\{0\right\} $. Then, the following claims are mutually equivalent:
	\begin{enumerate}[label={(\alph*)}]
		\item\label{lem:038:1} The Poincar\'e series 
		\begin{equation}\label{eq:033}
			P_{\Gamma,\rho} f=\sum_{\gamma\in\Gamma}f\big|_\rho\gamma
		\end{equation}
		converges absolutely and uniformly on compact subsets of $ \mathcal D $.
		\item\label{lem:038:2} The Poincar\'e series
		\begin{equation}\label{eq:034}
			P_\Gamma F_f=\sum_{\gamma\in\Gamma}F_f(\gamma\spacedcdot)
		\end{equation}
		converges absolutely and uniformly on compact subsets of $ G $.
		\item\label{lem:038:3} The Poincar\'e series
			\[ P_\Gamma v^*F_f=\sum_{\gamma\in\Gamma}\scal{F_f(\gamma\spacedcdot)}v_V \]
		converges absolutely and uniformly on compact subsets of $ G $.
	\end{enumerate} 
	Moreover, assuming \ref{lem:038:1}--\ref{lem:038:3}, we have
	\begin{equation}\label{eq:031}
		F_{P_{\Gamma,\rho} f}=P_\Gamma F_f.
	\end{equation}
\end{lemma}

\begin{proof}
	For every $ g\in G $, we have
	\[ 
	\begin{aligned}
		\left(P_\Gamma F_f\right)(g)&\overset{\phantom{\eqref{eq:030}}}=\sum_{\gamma\in\Gamma}F_f(\gamma g)\overset{\eqref{eq:032}}=\sum_{\gamma\in\Gamma}\left(f\big|_\rho\gamma g\right)(0)=\sum_{\gamma\in\Gamma}\left(f\big|_\rho\gamma \big|_\rho g\right)(0)\\
		&\overset{\eqref{eq:030}}=\rho(g_0)^{-1}\sum_{\gamma\in\Gamma}\left(f\big|_\rho\gamma\right)(g.0)=\rho(g_0)^{-1}(P_{\Gamma,\rho} f)(g.0).
	\end{aligned}
	\]
	This implies the equivalence of claims \ref{lem:038:1} and \ref{lem:038:2} since the assignment $ gK\mapsto g.0 $ defines a homeomorphism $ G/K\to\calD $. Next, \ref{lem:038:2} trivially implies \ref{lem:038:3}. Conversely, if \ref{lem:038:3} holds, then for every $ k\in K $, the series
	\[ \begin{aligned}
		\sum_{\gamma\in\Gamma}\scal{F_f(\gamma\spacedcdot)}{\rho(k)v}_V
		&=\sum_{\gamma\in\Gamma}\scal{\rho(k)^{-1}F_f(\gamma\spacedcdot)}{v}_V\overset{\text{Lem.\,\ref{lem:036}}}=\sum_{\gamma\in\Gamma}\scal{F_f(\gamma\spacedcdot k)}{v}_V\\&=\left(P_\Gamma v^*F_f\right)(\spacedcdot k)
	\end{aligned} \]
	converges absolutely and uniformly on compact subsets of $ G $, which implies \ref{lem:038:2} since the vectors $ \rho(k)v $ span $ V $. Finally, \eqref{eq:031} follows directly from \eqref{eq:032}, \eqref{eq:033}, and \eqref{eq:034}. 
\end{proof}

Theorem \ref{thm:100} is subsumed within the following theorem.

\begin{theorem}\label{thm:100a}
	Suppose that the representation $ \pi_\rho $ is integrable (see Lemma \ref{lem:022}\ref{lem:022:6}).
	Let $ \Gamma $ be a discrete subgroup of finite covolume in $ G $ such that $ m_\Gamma(\pi_\rho)<\infty $. Then, for every $ f\in\mathcal P(\mathcal D,V) $, the Poincar\'e series $ P_{\Gamma,\rho}f $ (resp., $ P_\Gamma v_\Lambda^*F_f $) converges absolutely and uniformly on compact subsets of $ \mathcal D $ (resp., $ G $), and we have
	\begin{equation}\label{eq:038}
		\calH_\rho^\infty(\Gamma)=\left\{P_{\Gamma,\rho}f:f\in\mathcal P(\mathcal D,V)\right\}
	\end{equation}
	and
	\begin{equation}\label{eq:151}
		\left(L^2(\Gamma\backslash G)_{[\pi_\rho^*]}\right)_{-\Lambda}=\left\{P_{\Gamma}v_\Lambda^*F_f:f\in\mathcal P(\mathcal D,V)\right\}.
	\end{equation}
\end{theorem}

\begin{proof}
	Let $ f\in \mathcal P(\mathcal D,V) $. By \eqref{eq:043} and Proposition \ref{prop:009}\ref{prop:009:2}, the $ K $-finite matrix coefficient $ c_{(\check F_{v_\Lambda})^*,(\check F_f)^*} $ of $ \calR_\rho^*\cong\left(\pi_\rho^*,\calH_\rho^*\right) $ is given by 
	\begin{equation}\label{eq:112}
		c_{(\check F_{v_\Lambda})^*,(\check F_f)^*}=C_\rho\, v_\Lambda^*F_f.
	\end{equation}
	Thus, by Lemmas \ref{lem:028} and \ref{lem:022}\ref{lem:022:6}, the series
	\begin{equation}\label{eq:150}
		P_\Gamma v_\Lambda^*F_f\overset{\eqref{eq:112}}=C_\rho^{-1}P_\Gamma c_{(\check F_{v_\Lambda})^*,(\check F_f)^*} 
	\end{equation}
	converges absolutely and uniformly on compact subsets of $ G $, hence by Lemma \ref{lem:038} the series $ P_{\Gamma,\rho}f $ converges absolutely and uniformly on compact subsets of $ \mathcal D $. 
	
	Next, since $ m_\Gamma(\pi_\rho^*)=m_\Gamma(\pi_\rho)<\infty $, by Lemma \ref{lem:022}\ref{lem:022:2a} the $ (\mathfrak g,K) $-module  
	\[ \left(L^2(\Gamma\backslash G)_{[\pi_\rho^*]}\right)_K\overset{\text{Prop.\,\ref{prop:027}}}=\mathrm{span}_\C\left\{P_\Gamma c_{h,h'}:h,h'\in(\calR_\rho^*)_K\right\} \]	
	decomposes into a direct sum of its weight subspaces, which are, by the $ (\mathfrak g,K) $-equivariance of the assignment $ h\mapsto P_\Gamma c_{h,h'} $, given by
	\begin{equation}\label{eq:039}
		\left(L^2(\Gamma\backslash G)_{[\pi_\rho^*]}\right)_{\lambda}=\mathrm{span}_\C\left\{P_\Gamma c_{h,h'}:h\in(\calR_\rho^*)_\lambda,h'\in(\calR_\rho^*)_K\right\},
	\end{equation}
	where $ \lambda $ goes over the weights of $ (\pi_\rho^*,\calH_\rho^*)\cong\calR_\rho^* $. In the case when $ \lambda=-\Lambda $, by \eqref{eq:042} and \eqref{eq:043} the equality \eqref{eq:039} states that $ \left(L^2(\Gamma\backslash G)_{[\pi_\rho^*]}\right)_{-\Lambda} $ is the space of functions
	\[ P_\Gamma c_{(\check F_{v_\Lambda})^*,(\check F_f)^*}
	\overset{\eqref{eq:150}}=C_\rho P_\Gamma v_\Lambda^*F_f=C_\rho v_\Lambda^*P_\Gamma F_f\overset{\eqref{eq:031}}=C_\rho v_\Lambda^*F_{P_{\Gamma,\rho}f}, \]
	where $ f\in\mathcal P(\mathcal D,V) $. This proves \eqref{eq:151}, and \eqref{eq:038} follows by applying the inverse of the linear isomorphism of Theorem \ref{thm:019} using Lemmas \ref{lem:036} and \ref{lem:012}.
\end{proof}

\section{A non-vanishing result}\label{sec:154}

In this section, we study the non-vanishing of Poincaré series $ P_{\Gamma,\rho}f\in\calH_\rho^\infty(\Gamma) $ of Theorem \ref{thm:100a} and the corresponding automorphic forms $ P_\Gamma v_\Lambda^*F_f\in \left(L^2(\Gamma\backslash G)_{[\pi_\rho^*]}\right)_{-\Lambda} $ using the following vector-valued version of Mui\'c's integral non-vanishing criterion \cite[Theorem 4.1]{muic09}.

\begin{proposition}\label{prop:101}
	Let $ \mathbf G $ be a unimodular, second-countable, locally compact Hausdorff group. Let $ \Gamma $ a discrete subgroup of $  \mathbf G $, and let $ E $ be a complex finite-dimensional Hilbert space.
	Let $ \varphi\in L^1(\mathbf G,E) $. 
	Then, we have the following:
	\begin{enumerate}[label={(\roman*)}]
		\item\label{prop:101:1} The Poincar\'e series 
		\[ \left(P_{\Gamma}\varphi\right)(g)=\sum_{\gamma\in\Gamma}\varphi(\gamma g) \]
		converges absolutely almost everywhere on $  \mathbf G $ and defines an element of $ L^1(\Gamma\backslash \mathbf G,E) $.
		\item\label{prop:101:2} Suppose that there exists a Borel set $ C\subseteq\mathbf G $ with the following properties:
		\begin{enumerate}[label=\textup{(C\arabic*)}]
			\item\label{enum:101:2:1} $ CC^{-1}\cap\Gamma=\left\{1_{ \mathbf G}\right\} $.
			\item\label{enum:101:2:2} $ \int_{C}\norm{\varphi(g)}_E\,dg>\frac12\int_{ \mathbf G}\norm{\varphi(g)}_E\,dg $.
		\end{enumerate}
		Then, we have \[ P_{\Gamma}\varphi\in L^1(\Gamma\backslash \mathbf G,E)\setminus\{0\}. \]
	\end{enumerate}
\end{proposition}

\begin{proof}
	A proof is obtained by replacing the absolute value symbol $ \left|\spacedcdot\right| $ by $ \norm{\spacedcdot}_E $ in the proofs of the scalar-valued versions of these results given in \cite[beginning of \S4]{muic09} (resp., \cite[Theorem 5.3]{zunar19}).
\end{proof}

Let $ \mathfrak a $ be a maximal abelian subspace of $ \mathfrak p $. We fix a choice of the set $ \Sigma^+ $ of positive restricted roots of $ \mathfrak g $ with respect to $ \mathfrak a $ (cf.\ \cite[Ch.\,VI, \S4]{knapp02}) and let
\[ \mathfrak a^+=\left\{X\in\mathfrak a:\lambda(X)>0\text{ for all }\lambda\in\Sigma^+\right\}. \]
We denote by $ m_\lambda $ the multiplicity of a restricted root $ \lambda\in\Sigma^+ $, i.e., the (real) dimension of $ \mathfrak g_{\lambda} $, the restricted root subspace of $ \mathfrak g $ corresponding to $ \lambda $. 

Let $ A^+ $ denote the closure of $ \exp(\mathfrak a^+) $ in $ G $. For every $ g\in G $, there exist $ k_1,k_2\in K $ and a unique $ a\in A^+ $ such that $ g=k_1ak_2  $ \cite[Ch.\,IX, Theorem 1.1]{helgason78}. Moreover, there exists $ M\in\R_{>0} $ such that
\[ M\int_G\varphi(g)\,dg=\int_K\int_{\mathfrak a^+}\int_K\varphi(k_1\exp(X)\,k_2)\left(\prod_{\lambda\in\Sigma^+}(\sinh\lambda(X))^{m_\lambda}\right)\,dk_1\,dX\,dk_2 \]
for all $ \varphi\in L^1(G) $ \cite[(2.14)]{flensted_jensen80} (see also \cite[Proposition 5.28]{knapp86}).
In particular, by Corollary \ref{cor:124} and \eqref{eq:125}, we have the following lemma. 

\begin{lemma}\label{lem:202}
	Suppose that the representation $ \pi_\rho $ is integrable (see Lemma \ref{lem:022}\ref{lem:022:6}). Then, for every $ f\in\mathcal P(\mathcal D,V) $ and for every Borel set $ S\subseteq\mathfrak a^+ $, the integral $ \int_{K\exp(S)K}\norm{F_f(g)}_V\,dg\in\R_{\geq0} $ is given by
		\[ M^{-1}\int_{S}\int_K\norm{F_f(k\exp(X))}_V\left(\prod_{\lambda\in\Sigma^+}(\sinh\lambda(X))^{m_\lambda}\right)\,dk\,dX.  \]
\end{lemma}

Our main non-vanishing result is the following theorem, which was stated in the Introduction as Theorem \ref{thm:103a}.

\begin{theorem}\label{thm:103}
	Suppose that the representation $ \pi_\rho $ is integrable (see Lemma \ref{lem:022}\ref{lem:022:6}).
	Let $ \Gamma $ be a discrete subgroup of finite covolume in $ G $ such that $ m_\Gamma(\pi_\rho)<\infty $. 
	Let $ f\in\mathcal P(\mathcal D,V)\setminus\left\{0\right\} $. 
	Suppose that there exists a Borel set $ S\subseteq\mathfrak a^+ $ with the following properties:
	\begin{enumerate}[label=\textup{(S\arabic*)}]
		\item\label{enum:103:1} $ K\exp(S)K\exp(-S)K\cap\Gamma=\left\{1_{G}\right\} $.
		\item\label{enum:103:2} We have
		\begin{equation}\label{eq:122}
			\begin{aligned}
				\int_S\int_K&\norm{F_f(k\exp(X))}_V\left(\prod_{\lambda\in\Sigma^+}(\sinh\lambda(X))^{m_\lambda}\right)\,dk\,dX \\
				&>\frac12\int_{\mathfrak a^+}\int_K\norm{F_f(k\exp(X))}_V\left(\prod_{\lambda\in\Sigma^+}(\sinh\lambda(X))^{m_\lambda}\right)\,dk\,dX.
			\end{aligned}  
		\end{equation}
	\end{enumerate}
	Then, we have
	\[ P_{\Gamma,\rho}f\in\calH_\rho^\infty(\Gamma)\setminus\left\{0\right\} \] and \[ P_{\Gamma}v_\Lambda^*F_f\in \left(L^2(\Gamma\backslash G)_{[\pi_\rho^*]}\right)_{-\Lambda}\setminus\left\{0\right\}. \]
\end{theorem}

\begin{proof}
	We apply Proposition \ref{prop:101} with $ \mathbf G=G $, $ E=V $, $ \varphi=F_f $, and $ C=K\exp(S)K $, noting that in this case the conditions \ref{enum:101:2:1} and \ref{enum:101:2:2} of Proposition \ref{prop:101}\ref{prop:101:2} are equivalent, respectively, to the conditions \ref{enum:103:1} and \ref{enum:103:2} by Lemma \ref{lem:202}. We conclude that $ P_\Gamma F_f\in L^1(\Gamma\backslash G,V)\setminus\left\{0\right\} $. By \eqref{eq:031}, Theorem \ref{thm:100a}, and Lemma \ref{lem:036}, it follows that $ P_{\Gamma,\rho}f\in\calH_\rho^\infty(\Gamma)\setminus\left\{0\right\} $. Consequently, the function $ P_{\Gamma}v_\Lambda^*F_f=v_\Lambda^*P_{\Gamma}F_f\overset{\eqref{eq:031}}=v_\Lambda^*F_{P_{\Gamma,\rho}f} $
 	belongs to $ \left(L^2(\Gamma\backslash G)_{[\pi_\rho^*]}\right)_{-\Lambda}\setminus\left\{0\right\} $ by Theorem \ref{thm:019}.
\end{proof}

\section{The case of bounded symmetric domains of type $ \textbf{\textit{A\,III}}$}\label{sec:160}

In this section, we make the theory of preceding sections explicit and give an example application of Theorem \ref{thm:103} in the case when the bounded symmetric domain $ \mathcal D $ is of type $ \textbf{\textit{A\,III}}$ in the notation of \cite[Ch.\,X, Table V]{helgason78}, i.e., is biholomorphic to the irreducible Hermitian symmetric space
\[ \SU(p,q)/\mathrm S(\U(p)\times\U(q)) \]
for some fixed integers $ p $ and $ q $ such that $ p\geq q\geq1 $. 

Writing $ n=p+q $ and $ I_{p,q}=\begin{pmatrix}I_p\\&-I_q\end{pmatrix} $, we specify the main objects of the preceding sections as follows. We set
\[ 
	G=\SU(p,q)=\left\{g\in\SL_{n}(\C):g^*I_{p,q}g=I_{p,q}\right\},
\] 
$ K=\mathrm S(\U(p)\times\U(q))=G\cap\U(n) $, and $ G_\C=\SL_{n}(\C) $. Thus, we have
\[ \mathfrak g=\mathfrak{su}(p,q)=\left\{\begin{pmatrix}A&B\\B^*&D\end{pmatrix}\in\mathfrak{sl}_{n}(\C):A\in M_p(\C),\ A^*=-A,\ D^*=-D\right\}, \]
and $ \mathfrak k $ is the subalgebra of anti-Hermitian matrices in $ \mathfrak g $. We let $ \mathfrak p $ (resp., $ \mathfrak h $) be the space of Hermitian (resp., diagonal) matrices in $ \mathfrak g $. Thus, $ \mathfrak h_\C $ is the Lie algebra of diagonal matrices in $ \mathfrak{sl}_n(\C) $. For $ r=1,\ldots,n $, we define a linear functional $ e_r\in\mathfrak h_\C^* $, $ e_r(X)=X_{r,r} $, and set
\[ \begin{aligned}
	\Delta_K^+&=\left\{e_r-e_s:1\leq r<s\leq p\text{ or }p+1\leq r<s\leq n\right\},\\
	\Delta_n^+&=\left\{e_r-e_s:1\leq r\leq p\text{ and }p+1\leq s\leq n \right\}.
\end{aligned} \]

The groups $ K_\C $, $ P_\C^+ $, and $ P_\C^- $ consist, respectively, of the matrices $ \begin{pmatrix}A\\&D\end{pmatrix} $, $ \begin{pmatrix}I_p&B\\&I_q\end{pmatrix} $, and $ \begin{pmatrix}I_p\\C&I_q\end{pmatrix} $, where $ A\in\GL_p(\C) $, $ D\in\GL_q(\C) $, $ \det A\cdot\det D=1 $, $ B\in M_{p,q}(\C) $, and $ C\in M_{q,p}(\C) $. The factorization $ g=g_+g_0g_- $ of a matrix $ g=\begin{pmatrix}A&B\\C&D\end{pmatrix}\in P_\C^+K_\C P_\C^- $ is given by
\begin{equation}\label{eq:114}
	 g=\begin{pmatrix}I_p&BD^{-1}\\&I_q\end{pmatrix}\begin{pmatrix}A-BD^{-1}C\\&D\end{pmatrix}\begin{pmatrix}I_p\\D^{-1}C&I_q\end{pmatrix}. 
\end{equation}
Following \cite[Ch.\,VI, \S2]{knapp86}, we identify the bounded symmetric domain
\[ \mathcal D=\left\{\underline z:=\begin{pmatrix}\ \ &z\\&\ \ \end{pmatrix}\in\mathfrak{gl}_{n}(\C):z\in M_{p,q}(\C),\ I_q-z^*z>0\right\}\subseteq\mathfrak p_\C^+ \]
with
\[ \Omega=\left\{z\in M_{p,q}(\C): I_q-z^*z>0\right\} \]
acted upon by $ G $ by
\begin{equation}\label{eq:115}
	 g.z=(Az+B)(Cz+D)^{-1},\quad g=\begin{pmatrix}A&B\\C&D\end{pmatrix}\in G,\ z\in\Omega. 
\end{equation}

We fix a maximal abelian subspace
\[ \mathfrak a=\left\{H_t\coloneqq\begin{pmatrix}&t_{p\times q}\\t_{q\times p}&\end{pmatrix}:t\in\R^q\right\} \]
 of $ \mathfrak p $, where we use the notation \eqref{eq:163}. For $ r=1,\ldots,q $, let $ \varepsilon_r\in\mathfrak a^* $ such that $ \varepsilon_r(H_t)=t_r $ for all $ t\in\R^q $. We fix the set $ \Sigma^+ $ of positive restricted roots $ \lambda $ listed, along with their multiplicities $ m_\lambda $, in the following table, whose last row should be removed in the case when $ p=q $:
\begin{equation}\label{eq:118}
	\begin{tabular}{rl|c}
	\multicolumn{2}{c|}{$ \lambda\in\Sigma^+ $} & $ m_\lambda $\\ \hline
	$ \varepsilon_r-\varepsilon_s$,& $ 1\leq r<s\leq q $ & $ 2 $\\ 
	$ \varepsilon_r+\varepsilon_s$,&  $ 1\leq r<s\leq q $ & $ 2 $\\
	$ 2\varepsilon_r$,&  $ 1\leq r\leq q $ & $ 1 $\\
	$ \ \varepsilon_r$,&  $ 1\leq r\leq q $ & $ 2(p-q) $\\
\end{tabular}.
\end{equation}
We have $ \mathfrak a^+=\left\{H_t:\ t_1>t_2>\cdots>t_q>0\right\} $. Given $ R\in\R_{>0} $, let
\[ T_R=\left\{t\in\R^q:R>t_1>\cdots>t_q>0\right\}\quad\text{and}\quad S_R=\left\{H_t:t\in T_R\right\}\subseteq\mathfrak a^+. \]

\begin{lemma}\label{lem:116}
	Let $ R\in\R_{>0} $ and $ g\in K\exp(S_R)K\exp(-S_R)K $. Then,
	\[  \norm g^2<n\,\cosh4R.  \]
\end{lemma}

\begin{proof}
	We have $ g=k_1\exp(H_t)\begin{pmatrix}A\\&D\end{pmatrix}\exp(H_{-l})k_2 $ for some $ k_1,k_2\in K $, $ A\in\U(p) $, $ D\in\U(q) $, and $ t,l\in T_{R} $. Let $ D'=\begin{pmatrix}D\\&\phantom{D}\end{pmatrix}\in M_p(\C) $. We have
	\[ \begin{aligned}
		\norm g^2
		&=\norm{\begin{pmatrix}\cosh(t_{p\times p})&(\sinh t)_{p\times q}\\(\sinh t)_{q\times p}&\cosh(t_{q\times q})\end{pmatrix}\begin{pmatrix}A\\&D\end{pmatrix}\begin{pmatrix}\cosh(l_{p\times p})&-(\sinh l)_{p\times q}\\-(\sinh l)_{q\times p}&\cosh(l_{q\times q})\end{pmatrix}}^2\\
		&=\sum_{r,s=1}^p\Big(\abs{\cosh t_r\,A_{r,s}\,\cosh l_s-\sinh t_r\,D_{r,s}'\,\sinh l_s}^2\\[-7mm]
		&\phantom{=\sum_{r,s=1}^p\Big(\,}+\abs{\cosh t_r\,A_{r,s}\,\sinh l_s-\sinh t_r\,D_{r,s}'\,\cosh l_s}^2\\[-7mm]
		&\phantom{=\sum_{r,s=1}^p\Big(\,}+\abs{\sinh t_r\,A_{r,s}\,\cosh l_s-\cosh t_r\,D_{r,s}'\,\sinh l_s}^2\\[-7mm]
		&\phantom{=\sum_{r,s=1}^p\Big(\,}+\abs{\sinh t_r\,A_{r,s}\,\sinh l_s-\cosh t_r\,D_{r,s}'\,\cosh l_s}^2\Big),
	\end{aligned} \]
	where we write $ t_r=l_r=0 $ for $ r>q $. One checks easily that the right-hand side equals
	\[ \begin{aligned}
		&\frac12\sum_{r,s=1}^p\left(\abs{A_{r,s}-D_{r,s}'}^2\cosh2(t_r+l_s)+\abs{A_{r,s}+D_{r,s}'}^2\cosh2(t_r-l_s)\right)\\
		&<\frac{\cosh4R}2\left(\norm{A-D'}^2+\norm{A+D'}^2\right)
		=\left(\norm{A}^2+\norm{D}^2\right)\cosh4R\\
		&=n\cosh4R.
	\end{aligned} \]
	This finishes the proof of the lemma.
\end{proof}

	In the example application of our non-vanishing result in Theorem \ref{thm:112} below, the role of $ \Gamma $ will be played by the arithmetic subgroup
	\[ \Gamma_G(1)=\SU(p,q;\Z[i])=G\cap\SL_n(\Z[i]) \]
	of $ G $ and its ``principal congruence subgroups''
	\[ \Gamma_G(N)=G\cap\big(I_{n}+M_n(N\Z[i])\big) \]
	of various levels $ N\in\Z_{>0} $. 

	\begin{lemma}\label{lem:115}
		Let $ N\in\Z_{>0} $.
		Let $ \gamma\in\Gamma_G(N)\setminus K $. Then, we have
		\[ \norm\gamma^2\geq 4N^2+n. \]
	\end{lemma}
	
	\begin{proof}
		We have $ \gamma=\begin{pmatrix}A&B\\C&D\end{pmatrix} $ for some $ A\in M_p(\Z[i]) $, $ B\in M_{p,q}(N\Z[i]) $, $ C\in M_{q,p}(N\Z[i]) $, and $ D\in M_{q}(\Z[i]) $.  Since $ \gamma\in G\setminus K $, we have  $ B\neq0 $ or $ C\neq0 $. Moreover, the equalities $ \gamma^*I_{p,q}\gamma=\gamma I_{p,q}\gamma^*=I_{p,q} $ imply that $ D^*D-B^*B=I_q $ and $ AA^*-BB^*=A^*A-C^*C=I_p $. Thus, $ A\notin\U(p) $, and both $ B $ and $ C $ are non-zero matrices with coefficients in $ N\Z[i] $. In particular, $ \norm B^2\geq N^2 $, $ \norm C^2\geq N^2 $, and $ \abs{b_{r,s}}\geq N $ for some $ r\in\left\{1,\ldots,p\right\} $ and $ s\in\left\{1,\ldots,q\right\} $, hence $ \left(BB^*\right)_{r,r}\geq N^2 $ and $ \left(B^*B\right)_{s,s}\geq N^2 $. Consequently, $ \left(AA^*\right)_{r,r}=\left(BB^*\right)_{r,r}+1\geq N^2+1 $ and $ \left(D^*D\right)_{s,s}=\left(B^*B\right)_{s,s}+1\geq N^2+1 $. Since moreover the equalities $ AA^*=BB^*+I_p $ and $ D^*D=B^*B+I_q $ imply that $ \left(AA^*\right)_{k,k}\geq1 $ for $ k=1,\ldots,p $ and $ \left(D^*D\right)_{k,k}\geq1 $ for $ k=1,\ldots,q $, it follows that $ \norm A^2=\sum_{k=1}^p\left(AA^*\right)_{k,k}\geq N^2+p $ and $ \norm D^2=\sum_{k=1}^q\left(D^*D\right)_{k,k}\geq N^2+q $. Thus, $ \norm\gamma^2=\norm A^2+\norm B^2+\norm C^2+\norm D^2\geq 4N^2+n $.
	\end{proof}

	Let us define the functions $ \mu_{p,q}:T_1\to\R_{>0} $,
	\[ \mu_{p,q}(x)=\left(\prod_{r=1}^q\frac{x_r^{p-q}}{(1-x_r)^{n}}\right)\prod_{1\leq r<s\leq q}(x_r-x_s)^2, \]
	and $ \nu_n:\Z_{>0}\to[0,1] $,
	\begin{equation}\label{eq:155}
		\nu_n(N)=\left(\sqrt{1+\frac n{2N^2}}+\sqrt{\frac n{2N^2}}\right)^{-2}.  
	\end{equation}

\begin{theorem}\label{thm:112}
	Let $ (\rho,V) $ be an irreducible unitary representation of $ K $ of highest weight $ \Lambda=\sum_{r=1}^{n-1}a_re_r $, where $ a_1,\ldots,a_{n-1}\in\Z $, $ 1-2n\geq a_1\geq\ldots\geq a_p $, and $ a_{p+1}\geq\ldots\geq a_{n-1}\geq0 $. Then, extending $ \rho $ to an irreducible holomorphic representation of $ K_\C $, we have the following:
	\begin{enumerate}[label={(\roman*)}]
		\item\label{thm:112:1} Let $ \Gamma $ be a discrete subgroup of finite covolume in $ G $ such that $ m_\Gamma(\pi_\rho)<\infty $ (e.g., $ \Gamma=\Gamma_G(N) $ for some $ N\in\Z_{>0} $). Then, for every $ f\in\mathcal P(\calD,V) $, the Poincar\'e series
		\begin{equation}\label{eq:205}
			\left(P_{\Gamma,\rho}f\right)\left(\hspace{.1mm}\underline z\hspace{.1mm}\right)=\sum_{\gamma=\left(\begin{smallmatrix}A&B\\C&D\end{smallmatrix}\right)\in\Gamma}\rho\begin{pmatrix}A-(\gamma.z)C\\&Cz+D\end{pmatrix}^{-1}f\left(\hspace{.2mm}\underline{\gamma.z}\hspace{.2mm}\right),\quad z\in\Omega,
		\end{equation}
		(resp., $ P_{\Gamma}v_\Lambda^*F_f $)
		converges absolutely and uniformly on compact subsets of $ \calD $ (resp., $ G $), and we have \eqref{eq:038} and \eqref{eq:151}.
		\item\label{thm:112:2} Let $ f\in\mathcal P(\calD,V)\setminus\left\{0\right\} $. Then, there exists the smallest integer $ N\geq3 $, which we will denote by $ N_0(\rho,f) $, such that
		\begin{equation}\label{eq:113}
			\int_{T_{\nu_{n}(N)}}\int_K\varphi_{\rho,f}(k,x)\,\mu_{p,q}(x)\,dk\,dx>\frac12\int_{T_{1}}\int_K\varphi_{\rho,f}(k,x)\,\mu_{p,q}(x)\,dk\,dx,
		\end{equation}
		where the function $ \varphi_{\rho,f}:K\times T_{1}\to\R_{\geq0} $ is given by
		\begin{equation}\label{eq:206}
			 \varphi_{\rho,f}(k,x)=\norm{\rho\begin{pmatrix}(I_p-x_{p\times p})^{-\frac12}\\&(I_q-x_{q\times q})^{\frac12}\end{pmatrix}\rho(k)^{-1}f\left(A\left(x^{\frac12}\right)_{p\times q}D^*\right)}_V, 
		\end{equation}
		where $ A\in\U(p) $ and $ D\in\U(q) $ are such that $ k=\begin{pmatrix}A\\&D\end{pmatrix} $.
		\item\label{thm:112:3} Let $ f\in\mathcal P(\calD,V)\setminus\left\{0\right\} $. Then, we have
		\[ P_{\Gamma_G(N),\rho}f\in\calH_\rho^\infty(\Gamma_G(N))\setminus\left\{0\right\} \]
		and
		\[ P_{\Gamma_G(N)}v_\Lambda^*F_f\in \left(L^2(\Gamma_G(N)\backslash G)_{[\pi_\rho^*]}\right)_{-\Lambda}\setminus\left\{0\right\} \]
		for all integers $ N\geq N_0(\rho,f) $.
	\end{enumerate}
\end{theorem}

\begin{proof}
	The assumption on $ \Lambda\in\mathfrak h_\C^* $ in the statement of the theorem is easily shown equivalent to the conditions \ref{enum:002:1}--\ref{enum:002:3} and \eqref{eq:110}.
	Thus, the representation $ \pi_\rho $ is well-defined and integrable. Moreover, the expression \eqref{eq:205} for the Poincar\'e series $ P_{\Gamma,\rho}f $ defined by \eqref{eq:033} follows from \eqref{eq:030} since
	\[ (\gamma\exp \underline z\hspace{.1mm})_0=\left(\begin{pmatrix}A&B\\C&D\end{pmatrix}\begin{pmatrix}I_p&z\\&I_q\end{pmatrix}\right)_0\overset{\eqref{eq:114}}=\begin{pmatrix}A-(\gamma.z)C\\&Cz+D\end{pmatrix} \]
	for all $ \gamma=\begin{pmatrix}A&B\\C&D\end{pmatrix}\in\Gamma $ and $ z\in\Omega $. Thus, \ref{thm:112:1} follows directly from Theorem \ref{thm:100a}. 
	
	Next, we apply Theorem \ref{thm:103} with $ S=S_R $ for a suitable $ R\in\R_{>0} $. For every $ N\in\Z_{\geq3} $, we have $ \Gamma_G(N)\cap K=\left\{1_G\right\} $, hence by Lemmas \ref{lem:116} and \ref{lem:115}, the set $ S_R $ satisfies the condition \ref{enum:103:1} of Theorem \ref{thm:103} with $ \Gamma=\Gamma_G(N) $ provided that
	\[ 4N^2+n\geq n\cosh4R,\qquad\text{i.e.,}\qquad R\leq \artanh\sqrt{\nu_n(N)}. \]
	
	Next, we consider the condition \ref{enum:103:2} of Theorem \ref{thm:103} for a fixed $ f\in\mathcal P(\calD,V)\setminus\left\{0\right\} $. By \eqref{eq:032}, \eqref{eq:114}, \eqref{eq:115}, and \eqref{eq:118}, for every $ N\in\Z_{>0} $, writing $ k=\begin{pmatrix}A\\&D\end{pmatrix} $ for $ k\in K $ as in \eqref{eq:206}, we have
	\begin{equation}\label{eq:120}
		 \begin{aligned}
		&\int_{S_{\artanh\sqrt{\nu_n(N)}}}\int_K\norm{F_f(k\exp(X))}_V\left(\prod_{\lambda\in\Sigma^+}(\sinh\lambda(X))^{m_\lambda}\right)\,dk\,dX\\
		&=\int_{T_{\artanh\sqrt{\nu_n(N)}}}\int_K\norm{\rho\begin{pmatrix}\cosh t_{p\times p}&\\&(\cosh t_{q\times q})^{-1}\end{pmatrix}\rho(k)^{-1}f\left(A(\tanh t)_{p\times q}D^*\right)}_V\\
		&\phantom{=\int_{T_{R}}}\cdot\left(\prod_{1\leq r<s\leq q}\sinh^2(t_r-t_s)\,\sinh^2(t_r+t_s)\right)\hspace{-1mm}\left(\prod_{r=1}^q\sinh(2t_r)\,\sinh^{2(p-q)}t_r\right)dk\,dt\\
		&=\int_{T_{\nu_n(N)}}\int_K\varphi_{\rho,f}(k,x)\,\mu_{p,q}(x)\,dk\,dx,
	\end{aligned} 
	\end{equation}
	where the last equality is obtained by introducing the change of variables $ x=\tanh^2(t) $. Analogously, we obtain
	\begin{equation}\label{eq:121}
		 \begin{aligned}
		\int_{\mathfrak a^+}\int_K\norm{F_f(k\exp(X))}_V\left(\prod_{\lambda\in\Sigma^+}(\sinh\lambda(X))^{m_\lambda}\right)\,dk\,dX
		=\int_{T_{1}}\int_K\varphi_{\rho,f}(k,x)\,\mu_{p,q}(x)\,dk\,dx.
	\end{aligned} 
	\end{equation}
	Since the set $ T_1 $ equals the increasing union $ \bigcup_{N=1}^\infty T_{\nu_n(N)} $, by the monotone convergence theorem the integral \eqref{eq:120} tends to \eqref{eq:121} as $ N\to\infty $. Moreover, since $ f\in\mathcal P(\calD,V)\setminus\left\{0\right\} $, by Corollary \ref{cor:124} and Lemma \ref{lem:036} the function $ F_f $ belongs to $ L^1(G,V)\setminus\left\{0\right\} $, hence by Lemma \ref{lem:202} the integral \eqref{eq:121} belongs to $ \R_{>0} $. Thus, \ref{thm:112:2} holds, and the inequality \eqref{eq:113} is satisfied by all integers $ N\geq N_0(\rho,f) $. By \eqref{eq:120} and \eqref{eq:121}, so is the condition \ref{enum:103:2} of Theorem \ref{thm:103} with $ S=S_{\artanh\sqrt{\nu_n(N)}} $. 
	
	Thus, by applying Theorem \ref{thm:103} with $ S=S_{\artanh\sqrt{\nu_n(N)}} $ and $ \Gamma=\Gamma_G(N) $, where $ N\in\Z_{\geq N_0(\rho,f)} $, we obtain \ref{thm:112:3}.
\end{proof}

\begin{example}\label{ex:162}
	Let $ m\in\Z_{\geq2n-1} $. An irreducible unitary representation $ \chi_m $ of $ K $ of highest weight $ -m\sum_{r=1}^pe_r $ is, upon extension to an irreducible holomorphic representation of $ K_\C $, given by
	\[ \chi_m(k)=(\det D)^m,\qquad k=\begin{pmatrix}A\\&D\end{pmatrix}\in K_\C. \]
	Suppose that $ p=q $. For every $ l\in\Z_{\geq0} $, we define a function $ f_{l}\in\mathcal P(\calD,\C) $ by
	\[ f_{l}(\underline z)=(\det z)^l,\quad z\in\Omega. \]
	By applying Theorem \ref{thm:112} to the Poincar\'e series $ P_{\Gamma_G(N),\chi_m}f_l $, we obtain the following non-vanishing result.
	
	\newpage
		
	\begin{corollary}\label{cor:161}
		Suppose that $ p=q $. Let $ m\in\Z_{\geq2n-1} $ and $ l\in\Z_{\geq0} $. Then, we have the following:
		\begin{enumerate}[label=(\roman*)]
			\item The integer $ N_0(\chi_m,f_l) $ defined in Theorem \ref{thm:112}\ref{thm:112:2} is the smallest integer $ N\geq3 $ such that 
			\[ \begin{aligned}
				\int_{T_{\nu_{n}(N)}}&\left(\prod_{r=1}^p x_r^{\frac{l}2}(1-x_r)^{\frac{m}2-n}\right)\left(\prod_{1\leq r<s\leq p}(x_r-x_s)^2\right)\,dx\\
				&>\frac12\int_{T_{1}}\left(\prod_{r=1}^p x_r^{\frac{l}2}(1-x_r)^{\frac{m}2-n}\right)\left(\prod_{1\leq r<s\leq p}(x_r-x_s)^2\right)\,dx.
			\end{aligned} \]
			\item For all integers $ N\geq N_0(\chi_m,f_l) $, the Poincar\'e series
			\[ 
			\left(P_{\Gamma_G(N),\chi_m}f_l\right)\left(\hspace{.1mm}\underline z\hspace{.1mm}\right)
			=\sum_{\left(\begin{smallmatrix}A&B\\C&D\end{smallmatrix}\right)\in\Gamma_G(N)}\frac{\det(Az+B)^l}{\det(Cz+D)^{l+m}},\qquad z\in\Omega, \]
			which converges absolutely and uniformly on compact subsets of $ \calD $, defines an element of  $ \calH_{\chi_m}^\infty(\Gamma_G(N))\setminus\left\{0\right\} $.
		\end{enumerate}
	\end{corollary}
	
	In Tables \ref{table:101} and \ref{table:102}, the values of integers $ N_0(\chi_m,f_l) $ are listed for $ p\in\left\{1,2\right\} $ and some choices of $ l $ and $ m $. We computed these values using Wolfram Mathematica 14.1 \cite{mathematica}; the code can be found in \cite{zunar25}. 
	
	\begin{table}[h!]
		\caption{Some values of $ N_0(\chi_m,f_l) $ in the case when $ p=1 $.}\label{table:101}
		\centering\footnotesize
		\begin{tabular}{|c|rrrrrrrrrrrrr|}
			\hline
			\backslashbox{$l$}{$m$}& \phantom{$ \geq $1\hspace{-1mm}}3 & \phantom{$ \geq $1\hspace{-1mm}}4 & \phantom{$ \geq $1\hspace{-1mm}}5 & \phantom{$ \geq $1\hspace{-1mm}}6 & \phantom{$ \geq $1\hspace{-1mm}}7 & \phantom{$ \geq $1\hspace{-1mm}}8 & \phantom{$ \geq $1\hspace{-1mm}}9 & \phantom{$ \geq $\hspace{-1mm}}10 & \phantom{$ \geq $\hspace{-1mm}}11 & \phantom{$ \geq $\hspace{-1mm}}12 & \phantom{$ \geq $\hspace{-1mm}}13 & \phantom{$ \geq $\hspace{-1mm}}14 & \phantom{$ \geq $}\hspace{-1mm}15  \\
			\hline
			0 & 7 & 3 & 3 & 3 & 3 & 3 & 3 & 3 & 3 & 3 & 3 & 3 & 3  \\
			1 & 12 & 5 & 3 & 3 & 3 & 3 & 3 & 3 & 3 & 3 & 3 & 3 & 3  \\
			2 & 16 & 6 & 4 & 3 & 3 & 3 & 3 & 3 & 3 & 3 & 3 & 3 & 3  \\
			3 & 20 & 8 & 5 & 4 & 3 & 3 & 3 & 3 & 3 & 3 & 3 & 3 & 3  \\
			4 & 25 & 9 & 6 & 5 & 4 & 3 & 3 & 3 & 3 & 3 & 3 & 3 & 3  \\
			5 & 29 & 11 & 7 & 5 & 4 & 4 & 3 & 3 & 3 & 3 & 3 & 3 & 3  \\
			6 & 34 & 12 & 8 & 6 & 5 & 4 & 4 & 3 & 3 & 3 & 3 & 3 & 3  \\
			7 & 38 & 13 & 8 & 6 & 5 & 4 & 4 & 4 & 3 & 3 & 3 & 3 & 3  \\
			8 & 42 & 15 & 9 & 7 & 6 & 5 & 4 & 4 & 4 & 3 & 3 & 3 & 3  \\
			9 & 47 & 16 & 10 & 8 & 6 & 5 & 5 & 4 & 4 & 4 & 3 & 3 & 3  \\
			10 & 51 & 18 & 11 & 8 & 7 & 6 & 5 & 4 & 4 & 4 & 4 & 3 & 3  \\
			11 & 56 & 19 & 12 & 9 & 7 & 6 & 5 & 5 & 4 & 4 & 4 & 4 & 3  \\
			12 & 60 & 21 & 13 & 9 & 8 & 6 & 6 & 5 & 5 & 4 & 4 & 4 & 3  \\
			\hline
		\end{tabular}
	\end{table}
	
	\begin{table}[h!]
		\caption{Some values of $ N_0(\chi_m,f_l) $ in the case when $ p=2 $.}\label{table:102}
		\centering\footnotesize
		\begin{tabular}{|c|rrrrrrrrrrrrrr|}
			\hline
			\backslashbox{$l$}{$m$} & \phantom{$ \geq $1\hspace{-1mm}}7 & \phantom{$ \geq $1\hspace{-1mm}}8 & \phantom{$ \geq $1\hspace{-1mm}}9 & \phantom{$ \geq $\hspace{-1mm}}10 & \phantom{$ \geq $\hspace{-1mm}}11 & \phantom{$ \geq $\hspace{-1mm}}12 & \phantom{$ \geq $\hspace{-1mm}}13 & \phantom{$ \geq $\hspace{-1mm}}14 & \phantom{$ \geq $\hspace{-1mm}}15  & \phantom{$ \geq $\hspace{-1mm}}16 & \phantom{$ \geq $\hspace{-1mm}}17 & \phantom{$ \geq $\hspace{-1mm}}18 & \phantom{$ \geq $\hspace{-1mm}}19 & \phantom{$ \geq $\hspace{-1mm}}20\\ \hline
			0 & 50 & 17 & 10 & 8 & 6 & 5 & 5 & 4 & 4 & 4 & 3 & 3 & 3 & 3\\
			1 & 64 & 21 & 13 & 9 & 7 & 6 & 5 & 5 & 4 & 4 & 4 & 4 & 3 & 3 \\
			2 & 77 & 25 & 15 & 11 & 9 & 7 & 6 & 6 & 5 & 5 & 4 & 4 & 4 & 4 \\
			3 & 91 & 29 & 17 & 12 & 10 & 8 & 7 & 6 & 6 & 5 & 5 & 4 & 4 & 4 \\
			4 & 105 & 33 & 19 & 14 & 11 & 9 & 8 & 7 & 6 & 6 & 5 & 5 & 5 & 4 \\
			5 & 119 & 37 & 22 & 15 & 12 & 10 & 8 & 7 & 7 & 6 & 6 & 5 & 5 & 5\\
			6 & 133 & 41 & 24 & 17 & 13 & 11 & 9 & 8 & 7 & 7 & 6 & 6 & 5 & 5 \\
			7 & 147 & 45 & 26 & 18 & 14 & 12 & 10 & 9 & 8 & 7 & 7 & 6 & 6 & 5 \\
			8 & 161 & 49 & 28 & 20 & 15 & 12 & 11 & 9 & 8 & 8 & 7 & 6 & 6 & 6 \\
			9 & 175 & 54 & 30 & 21 & 16 & 13 & 11 & 10 & 9 & 8 & 7 & 7 & 6 & 6 \\
			10 & 189 & 58 & 33 & 23 & 17 & 14 & 12 & 11 & 9 & 8 & 8 & 7 & 7 & 6 \\
			11 & 202 & 62 & 35 & 24 & 19 & 15 & 13 & 11 & 10 & 9 & 8 & 8 & 7 & 7 \\
			12 & 216 & 66 & 37 & 26 & 20 & 16 & 14 & 12 & 10 & 9 & 9 & 8 & 7 & 7 \\
			\hline
		\end{tabular}
	\end{table}

\end{example}

\newpage
 
\bibliographystyle{amsplain}
\linespread{.96}

\end{document}